\documentclass[11pt]{amsart}
\usepackage{amssymb,amsbsy,amsmath,amsfonts,amssymb,amscd,graphicx, stackrel}
\usepackage{color}

\newtheorem{theorem}{Theorem}[section]
\newtheorem{corollary}[theorem]{Corollary}
\newtheorem{lemma}[theorem]{Lemma}

\newtheorem{remark}[theorem]{Remark}
\numberwithin{equation}{section}

\newcommand{\R}{\mathbb{R}}

\newcommand{\eps}{\varepsilon}
\renewcommand{\phi}{\varphi}

\newcommand{\beq}{\begin{equation}}
\newcommand{\eeq}{\end{equation}}
\newcommand{\Email}[1]{\rm{\sl E-mail\/}: \textsf{#1}}
\newcommand{\al}{\alpha}
\newcommand{\ve}{\varepsilon}

\def\neweq#1{\begin{equation}\label{#1}}
\def\endeq{\end{equation}}
\def\eq#1{(\ref{#1})}
\def\proof{\noindent{\it Proof. }}
\def\endproof{\hfill $\Box$\par\vskip3mm}

\definecolor{darkblue}{rgb}{0.05, .05, .65}
\definecolor{darkgreen}{rgb}{0.05, .55, .05}
\definecolor{darkred}{rgb}{0.8,0,0}
\begin{document}

\title[Emden-Fowler equation on the hyperbolic space]{{\bf Classification of radial solutions to the\\[6pt]
 Emden-Fowler equation on the hyperbolic space}}

\author[M. Bonforte]{Matteo Bonforte}
\author[F. Gazzola]{Filippo Gazzola}
\author[G. Grillo]{Gabriele Grillo}
\author[J.L. V\'azquez]{Juan Luis V\'azquez}

\address{\hspace*{-30pt}
{\sc{M. Bonforte:}} \rm Departamento de Matem\'{a}ticas, Univ. Aut\'{o}noma de Madrid, Campus de Cantoblanco, 28049 Madrid, Spain. \Email{matteo.bonforte@uam.es}\newline
\hspace*{-15pt}{\sc{F. Gazzola:}} \rm Dip. di Matematica, Politecnico di Milano, Piazza Leonardo da Vinci 32, 20133 Milano, Italy. \Email{filippo.gazzola@polimi.it}\newline
\hspace*{-15pt}{\sc{G. Grillo:}} \rm Dip. di Matematica, Politecnico di Milano, Piazza Leonardo da Vinci 32, 20133 Milano, Italy. \Email{gabriele.grillo@polimi.it}\newline
\hspace*{-15pt}{\sc{J.L. V\'azquez:}} \rm Departamento de Matem\'{a}ticas, Univ. Aut\'{o}noma de Madrid, Campus de Cantoblanco, 28049 Madrid, Spain. \Email{juanluis.vazquez@uam.es}}

\begin{abstract}We study the Emden-Fowler equation $-\Delta u=|u|^{p-1}u$ on the hyperbolic space ${\mathbb H}^n$.
We are interested in radial solutions, namely solutions depending only on the geodesic distance from a given point. The critical exponent for such equation is $p=(n+2)/(n-2)$ as in the Euclidean setting, but the properties of the solutions show striking differences with the Euclidean case. While the papers \cite{mancini, bhakta} consider finite energy solutions, we shall deal here with infinite energy solutions and we determine the exact asymptotic behavior of wide classes of finite and infinite energy solutions.

\end{abstract}
\maketitle


\section{\bf Introduction}

In this paper we consider the following nonlinear  elliptic equation
\neweq{triangle}
\Delta u+|u|^{p-1}u=0\qquad\mbox{in }\mathbb H^n,
\endeq
on the simplest example of manifold with negative curvature, the hyperbolic space ${\mathbb H}^n$, in dimension $n\ge3$. $\Delta$ is the Laplace-Beltrami operator on $\mathbb{H}^n$ and we take $p>0$. When posed in the Euclidean space $\R^n$ this equation is known as the Emden-Fowler equation and
the study goes back to Lane \cite{Lane}, Emden \cite{Emden}, Fowler \cite{Fowler},  Chandrasekhar \cite{Chan} and others, and plays an important role in Astrophysics. Attention was focused on the existence and description  of radial solutions. There is a host of later important contributions to the subject; among them we must mention the famous paper by Joseph and Lundgren \cite{JoLu} where a complete classification of radial solutions is done. For more general nonlinear elliptic equations in the Euclidean space $\R^n$  we refer to \cite{BL,gaz,gnn,gnn2,NS1,NS2,PuSer}.

The existence of radial solutions of the Emden-Fowler equation can be addressed in the hyperbolic space in the setting of radial functions provided we define a function to be radial if it depends on the Riemannian distance $r$ from a pole $o$. We recall that several models can be used to describe ${\mathbb H}^n$ in an explicit coordinate system. For instance, one may realize ${\mathbb H}^n$ as an embedded hyperboloid in $\mathbb R^{n+1}$, endowed with the inherited metric. It is also possible to use one of the two Poincar\'e realizations, namely the ball model or the half--space model, in the sense that topologically one can identify ${\mathbb H}^n$ with the unit ball in ${\mathbb R}^n$ or with the upper half--space, each of which endowed with an appropriate metric with the property that the
Riemannian distance from any given point to points approaching the topological boundary tends to $+\infty$. Another possible realization is the Klein model, see \cite{bea, ratc} for a comprehensive discussion.
Because of the structure of the isometry group of $\mathbb H^n$ it is convenient to describe the
hyperbolic space as a {\it model manifold}, see \cite{gw} for details. On such a manifold, a pole $o$ is given and the metric has the form
\[
{\rm d}s^2={\rm d}r^2+f(r)^2{\rm d}\omega^2,
\]
for an appropriate function $f$, where $r$ is the Riemannian distance from the pole $o$ and ${\rm d}\omega^2$ denotes the canonical metric on the unit sphere. The hyperbolic space ${\mathbb H}^n$ is obtained by making the choice $f(r)=\sinh r$. It is then known, see \cite{D} and references therein, that the radial part of the Laplacian has the explicit expression, on radial functions $u$,
\beq\label{lapl}
\Delta_{rad}u(r)=u^{\prime\prime}(r)+(n-1)(\coth r)
u^\prime(r)=\frac1{(\sinh r)^{n-1}}\left[(\sinh r)^{n-1}u^\prime(r)\right]^\prime,
\eeq
and that in such coordinates the volume element is ${\rm d}\mu=(\sinh r)^{n-1}\,{\rm d}r\,{\rm d}\omega_{n-1}$, where ${\rm d}\omega_{n-1}$ is the volume element on the $(n-1)$--dimensional unit sphere $\mathbb{S}^{n-1}$.

 Our aim in this paper is to classify the smooth radial solutions to \eqref{triangle}, which  satisfy the ODE \neweq{equazione}
u''(r)+(n-1)(\coth r)u'(r)+|u(r)|^{p-1}u(r)=0\qquad \mbox{ for} \ r>0\,,
\endeq
together with the initial conditions
\begin{equation}\label{initcond}
u(0)=\alpha\,,\qquad u'(0)=0\,.
\end{equation}

\noindent The mathematical study of this problem was initiated in \cite{mancini,bhakta} for the slightly more general equation $\Delta u+\lambda u+|u|^{p-1}u=0$ in the range $p\in(1,\frac{n+2}{n-2})$ and {\em energy solutions} are  considered. An energy solution is a function in H$^1({\mathbb H}^n)$, which is the natural space where variational methods can be successfully employed.

\medskip

\noindent {\sc Results and methods.}
Here we  study the whole class of radial solutions to \eqref{triangle}--\eqref{initcond} and  consider all values of $p>0$. The study of non-variational solutions is quite natural both in the supercritical case $p\ge \frac{n+2}{n-2}$, where no radial solution belongs to the energy space, but also in the subcritical case $p\in(1,\frac{n+2}{n-2})$, where there exist infinitely many positive solutions to \eqref{triangle} which do not belong to the energy space, see \cite{mancini}.

We determine the intersection properties and the asymptotic behavior at infinity of all radial solutions. In the subcritical case we also prove that there are infinitely many sign-changing solutions of \eqref{triangle}
not included in the energy space; we determine again their asymptotic behavior as well as the asymptotics of the (sign-changing) energy solutions.
We also show that any sign-changing solution has finitely many oscillations, in striking contrast with the Euclidean behavior, see \cite{PuSer}.
Finally, in the sublinear case $0<p<1$, we prove that no positive radial solution to \eqref{triangle} exists and that all sign-changing solutions exhibit infinitely many oscillations.\par

For the proof of our results we construct a generalized Poho\v{z}aev-type functional (notice the alternative spelling Pokhozhaev)  \cite{poho}
in the hyperbolic setting. This construction requires a delicate choice of the test functions involved:
instead of powers of $r$ we use particular combinations of the hyperbolic functions. This functional gives several different information in the three cases $p\ge\frac{n-2}{n+2}$, $1<p<\frac{n-2}{n+2}$, and $0<p<1$. We refer to \cite{BDE, BrezNir, gaz, poho} for information on the roles played by these exponents in $\R^n$. The next step consists in adapting the techniques developed by Ni-Serrin \cite{NS1,NS2} to this new
framework. The decay rate of solutions to \eqref{equazione} follows by a careful reworking of the differential equation at hand combined with the Poho\v{z}aev-type functional.\par
The exact statements of our results are given in next section, at the end of which we shall also discuss briefly, for the sake of completeness, the linear case $p=1$.  In Section \ref{final} we complement our results by further remarks and open problems.

In Section \ref{pohofunctional} we introduce the Poho\v{z}aev-type functional which turns out to be very powerful in the study of the qualitative behavior of solutions to \eqref{equazione} for {\em any} value of $p$.
Section \ref{proofsuper} is devoted to the proofs of the results in the supercritical case.
Sections \ref{proofsub} and \ref{oscillation} deal, respectively, with the proofs of the results concerning positive solutions and sign-changing
solutions in the subcritical case. Finally, in Section \ref{proofsublin} we briefly deal with the sublinear case $p<1$. In these sections we shall also provide numerical simulations and plots of the qualitative properties of solutions.

We conclude this introduction with two remarks: (1) The case of the hyperbolic space with curvature $-c^2$, $c>0$ can be easily reduced to the case $c=1$ that we treat here in detail. Actually, for $c\ne 1$ the radial solutions of \eqref{triangle} satisfy an ODE of the same form
\neweq{equazione.c}
u''(r)+(n-1)c\,\coth (cr)\,u'(r)+|u(r)|^{p-1}u(r)=0\qquad \mbox{ for} \ r>0\,.
\endeq
The change of variables
\begin{equation}\label{UC.intro}
\overline{u}(r)=c^{q}\, u(cr), \quad q=2/(p-1)\,,
\end{equation}
transforms solutions $\overline{u}(r)$ of \eqref{equazione.c} into solutions $u(r)$  of \eqref{equazione}. Note that $c\,\coth (cr)\to 1/r$ as $c\to 0$, so we recover the Euclidean case  in that limit. We will comment later  on some consequences,   see Remark \ref{R(a)}, iii).

\noindent (2) We expect that the study of the elliptic problem \eqref{triangle} might be relevant for the study of the fine asymptotics of the solutions to the corresponding evolutionary problem $u_t=\Delta u^m$, with $1>m=\frac 1p$, as initiated in \cite{BGV1}, in the spirit of the recent results given in \cite{BGV2}. This same elliptic problem has been recently considered in \cite{CM} in the study of existence and stability of finite energy solitons for the subcritical nonlinear Schr\"odinger equation.
Although our results can be extended to the equation
\neweq{triangle2}
\Delta u+f(u)=0\qquad\mbox{in } \ \mathbb H^n,
\endeq
where $f$ satisfies suitable assumptions, in the present paper we limit ourselves to consider the particular case $f(u)=|u|^{p-1}u$.

\section{\bf Classification of radial solutions. Statement of results}\label{statements}

The results strongly depend on the exponent $p$ in three different ranges. When $p\ge\frac{n+2}{n-2}$ we say that $p$ is {\em supercritical}, when
$1<p<\frac{n+2}{n-2}$ we say that $p$ is \textit{subcritical}, when $0<p<1$ we say that $p$ is {\em sublinear}.

\medskip

\noindent $\bullet$ {\sc The supercritical case}. Our main result for this case is the following.

\begin{theorem}\label{supercritico}
For any $p\ge\frac{n+2}{n-2}$ equation \eqref{triangle} admits infinitely many positive radial solutions $u=u(r)$ and infinitely many negative solutions. In fact, all radial solutions $u$ to \eqref{triangle} with $u(0)>0$, $u'(0)=0$, are everywhere positive and decay polynomially at infinity with the following rate
\begin{equation}\label{ratesuper}
\lim_{r\to+\infty}r^{1/(p-1)}u(r)=c(n,p):=\left(\frac{n-1}{p-1}\right)^{1/(p-1)}
\end{equation}
and
\[
\lim_{r\to+\infty}\frac{u^\prime(r)}{u(r)}=\lim_{r\to+\infty}\frac{u^{\prime\prime}(r)}{u^\prime(r)}=0.
\]
On the other hand, for $u(0)<0$, $u'(0)=0$, the solutions are everywhere negative and decay polynomially with just the opposite limit $-c(n,p)$, in \eqref{ratesuper}. In particular, any such solution $u$ belongs to ${\rm L}^q(\mu)$ only for $q=\infty$.
\end{theorem}

This result is qualitatively similar to the Euclidean case, but the power-law decay determined in \eqref{ratesuper} is different. We recall that solutions in the
Euclidean case decay like $r^{-2/(p-1)}$, see \cite[Theorem 2.2]{NS2}.

As a byproduct of the proof of Theorem \ref{supercritico} we obtain the following non-existence result for
solutions to the Dirichlet problem in a ball:

\begin{corollary}\label{ball}
If $p\ge\frac{n+2}{n-2}$, then for any radius $R>0$, equation \eqref{triangle} admits no positive radial solution $u=u(r)$ satisfying $u(R)=0$.
\end{corollary}

This result may also be obtained by adapting the arguments in \cite{stap} valid for $p=\frac{n+2}{n-2}$ to the supercritical range $p\ge\frac{n+2}{n-2}$

\medskip

\noindent$\bullet$ {\sc The subcritical case}. A first main novelty of this case is the existence of
a positive global solution having fast decay at infinity (a hyperbolic nonlinear ground state). This type of solution has been obtained by Mancini-Sandeep \cite[Theorems 1.3-1.4, Lemma 3.4]{mancini} using variational methods in the space
$$H^1_r({\mathbb H}^n)=\{u\in {\rm L}^2(\mu);\, \nabla u\in {\rm L}^2(\mu),\, u=u(r)\},$$
where $\mu$ is the Riemannian measure, $\nabla$ is the Riemannian gradient and $r$ is the Riemannian distance from a given pole $o$.  We state their result for convenience.

\noindent\textbf{Theorem A} \cite[Theorems 1.3-1.4, Lemma 3.4]{mancini}\label{sottocritico}
{\sl Let $1<p<\frac{n+2}{n-2}$. There exists a unique function $U\in H^1_r({\mathbb H}^n)$ which is a radial, smooth, positive and bounded solution
to the equation \eqref{triangle}. The function $U$ is (radially) decreasing and there exists $\overline{c}>0$  such that
\beq\label{n-1}
\lim_{r\to+\infty}e^{(n-1)r}U(r)=\overline{c}.
\eeq}

Of course, there exists a unique negative ground state   which is given by $-U$.
We use the ground state $U$ in the classification of all radial solutions to \eqref{equazione}--\eqref{initcond}. Without loss of generality we restrict ourselves
to the case $u(0)=\alpha>0$.
One first class of radial solutions is given by the next result.

\begin{theorem}\label{polynomial}
Let $1<p<\frac{n+2}{n-2}$ and let $U$ be the positive ground state described in Theorem A. Each local solution $u$ to \eqref{equazione}--\eqref{initcond} satisfying
\neweq{shooting}
0<u(0)<U(0)
\endeq
can be extended as a positive solution for $0<r<\infty$, hence generating a positive radial solution to \eqref{triangle} on ${\mathbb H}^n$. Moreover, there exists a unique $r_0>0$ such
that $u(r_0)=U(r_0)$ and the asymptotic behavior is given by
\beq\label{pol}
\lim_{r\to+\infty}r^{1/(p-1)}u(r)=c(n,p)\,,
\eeq
the constant of \eqref{ratesuper}. None of these slow-decaying solutions belongs to the energy space.
\end{theorem}

The first part of the result is known from paper \cite{mancini}, and our contribution is the bound \eqref{pol}, which is exactly the same as in the supercritical case, and makes the solutions non-variational.

A second main difference with the supercritical case is the presence of sign-changing solutions, which we now discuss. Recall that in the Euclidean case, Pucci-Serrin \cite[Theorem 15]{PuSer} show that all sign-changing solutions in $\R^n$ have infinitely many zeros. We prove that this never happens in $\mathbb H^n$, namely any sign-changing radial solution to \eqref{triangle} has a finite number of zeros. Finally, we show that there exists infinitely many solutions of infinite energy.

\begin{theorem}\label{osc}
Let $p$ and  $U$ be as in Theorem \ref{polynomial}. If $u$ is a  solution to \eqref{equazione}--\eqref{initcond} with $u(0)>U(0)$, then it is sign-changing.
Moreover:

\noindent (i) the solution with initial data $\alpha$ vanishes for the first time at a finite point $r_\alpha$
and the map  $\alpha\mapsto r_\alpha$ is a monotone decreasing 1-1 map from $(U(0),\infty)$ onto $(0,\infty)$;

\noindent (ii) any radial sign-changing solution has finitely many zeros;

\noindent (iii) there exist infinitely many radial sign-changing solutions $u\not\in H^1(\mathbb{H}^n)$, having exactly one zero, and satisfying
\neweq{negativelimit}
\lim_{r\to+\infty}r^{1/(p-1)}u(r)=-c(n,p),
\endeq\nobreak
the constant of \eqref{ratesuper};

\noindent (iv) for any integer $k\ge1$ there exists infinitely many solutions to \eqref{equazione}--\eqref{initcond} having exactly $k$ zeros;

\noindent (v) any radial sign-changing solution $u\in H^1(\mathbb H^n)$ satisfies \eqref{n-1} for some real constant $\overline{c}$.
\end{theorem}

\begin{remark}\label{R(a)} \em i)  As a corollary of our results, we see that we can identify the solution $U$ with the separatrix between the sign-changing class from the globally positive radial solutions in hyperbolic space. In particular all radial solutions $u$ satisfying $u(0)>U(0)$ change sign.

\noindent ii) The L$^\infty$-norm $U(0)$  of the variational solution $U(r)$ is the optimal a priori bound for all positive radial and global solutions in the subcritical case. Sign-changing solutions have no a priori bound.

\noindent iii) When we work in the hyperbolic space with curvature $-c^2\ne -1$ the rescaling stated in \eqref{UC.intro} implies that the ground state is $U_c(r)=c^{2/(p-1)}U(cr)$. Therefore the a priori bound is
$$
M_c=\sup_{r\ge 0} U_c(r)= c^{2/(p-1)}U(0)\,,
$$
which goes to zero as $c\to 0$. This explains how the hyperbolic ground state disappears in the Euclidean limit.

\noindent iv) In fact, from our proof one sees that Item (iv) can be complemented with the statement that for any integer $k\ge1$ there exists $\alpha_k>0$ such
that if $u(0)>\alpha_k$, then the solution to \eqref{equazione}--\eqref{initcond} has at least $k$ zeros.

\end{remark}\rm

\medskip

As for previous results, it was proved in \cite[Proposition 4.4]{mancini} that the corresponding Dirichlet problem admits a unique radial positive solution
in any ball of finite radius. Moreover, Bhakta-Sandeep \cite[Theorem 5.1]{bhakta} showed that there exist infinitely many sign-changing solutions to \eqref{triangle}, which can be chosen to be radial and belonging to $H^1(\mathbb H^n)$ and that \cite[Theorem 4.2]{bhakta} any solution in the energy space, not necessarily radial, satisfies the bound
\beq\label{spurio}
|u(x)|\le C\,e^{-\frac{n-2}2r}
\eeq
where $r=\varrho(x,o)$, $o\in\mathbb H^n$ is the pole and $\varrho$ is the Riemannian distance. Moreover, the proof of \cite[Theorem 3.1]{bhakta}
shows that for radial solutions the upper bound \eqref{spurio} can be improved from $\frac{n-2}2$ to $\frac{n-1}2$. We show that such solutions
satisfy the stronger property \eqref{n-1} for a suitable real constant $\overline{c}$ and we discuss their oscillation features.

\medskip

\medskip

\noindent$\bullet$ {\sc The sublinear case}. In this case we have no globally positive solutions at all. Moreover all sign-changing solutions have
infinitely many zeros and ``slow'' decay at infinity in the sense that the bound $|u(r)|\le Ce^{-(n-1)r}$ cannot hold for all $r>0$ and a suitable $C>0$, contrary to \eqref{n-1}.

\begin{theorem}\label{sublinear}
Let $0<p<1$. Then there exists no positive radial solution to \eqref{triangle}. All radial solutions to \eqref{triangle} change sign infinitely many
times and
\neweq{supinf}
\limsup_{r\to+\infty}\ e^{\frac{n-1}{p+1}r}u(r)>0\ ,\qquad\liminf_{r\to+\infty}\ e^{\frac{n-1}{p+1}r}u(r)<0\ .
\endeq
\end{theorem}

\noindent $\bullet$ {\sc Some comments on the linear case.} Note that if $p\not=1$ and $u$ solves the equation $\Delta u+\,|u|^{p-1}u=0$ on
$\mathbb H^n$, then $v=c^{1/(1-p)}u$ solves the equation $\Delta v+c\,|v|^{p-1}v=0$. But, of course, this simplifying trick does not apply when $p=1$.
For completeness and comparison with the nonlinear cases $p\neq1$, we recall here some facts about the linear case $p=1$.

It is well-known \cite{bea, D} that the L$^2$-spectrum of $-\Delta$ on $\mathbb H^n$ coincides with the half-line $[\Lambda,+\infty)$ where $\Lambda=(n-1)^2/4$. The equation
\beq\label{eigen}
\Delta u+c u=0\qquad\mbox{in }\mathbb H^n,
\eeq
has a radial positive solution (generalized ground state) with exponential decay for $c=\Lambda$. In Section \ref{final} we show that any such solution $u$ is comparable
to $(1+r)e^{-\frac{n-1}2r}$ when $r\to+\infty$ so that, in particular, $u\not\in{\rm L}^2$.\par
If $0<c<\Lambda$, then radial solutions to \eqref{eigen} with $u(0)>0$ are positive and slowly decaying at infinity, this behavior bears some similarity with
solutions corresponding to small initial data in the subcritical case, see Theorem \ref{polynomial}. On the other hand, when $c>\Lambda$, it belongs to the
L$^2$-spectrum. In this case, radial solutions $u$ to \eq{eigen} change sign infinitely many times and $|u(r)|\le C(1+r)e^{-\frac{n-1}2\,\,r}$. This behavior has now
to be compared with the sublinear case, see Theorem \ref{sublinear}. Further details are given in Section \ref{final}.


\section{\bf Further remarks and open problems}\label{final}

\noindent $\bullet$ {\bf Functional analysis on the hyperbolic space.} $\mathbb{H}^n$ is a
non-compact manifold and, since the Ricci curvature is constant and negative and the space is simply connected, {\it both the Sobolev and the Poincar\'e inequality hold}. In other words,
denoting by $\mu$ the Riemannian measure, we have both the inequalities
\beq\label{Sobeuc} \int_{{\mathbb H}^n} |u|^{2n/(n-2)}\,{\rm d}\mu\le C_1\int_{{\mathbb H}^n}|\nabla
u|^2\,{\rm d}\mu \eeq
and
\beq\label{poincare}
\int_{{\mathbb H}^n}u^2\,{\rm d}\mu\le C_2\int_{{\mathbb H}^n}|\nabla u|^2\,{\rm d}\mu \eeq
for all $u\in H^1({\mathbb H}^n)$, the usual Sobolev space. For generalizations and sharp form of such inequalities see \cite{mancini, bhakta}. These properties are important in the variational analysis. \par

\noindent$\bullet$ {\bf Some explicit ground states.} There are at least three explicit ground state solutions in the subcritical case $1<p<(n+2)/(n-2)$. They are

$$
U(r)=\frac{[n^2(n-1)]^{n-1}}{(1+\cosh r)^{n-1}}\qquad\mbox{for }p=\frac n{n-1}
$$
$$
U(r)=\frac{[n(n-1)]^{(n-1)/2}}{(\cosh r)^{n-1}}\qquad\mbox{for }p=\frac{n+1}{n-1}
$$
$$
U(r)=\left(\frac{n(n-1)}{n+1}\right)^{(n-1)/4}\frac1{\left((\cosh r)^2-\frac n{n+1}\right)^{(n-1)/2}}\qquad\mbox{for }p=\frac{n+3}{n-1}\,.
$$
They solve \eqref{triangle} and have the announced exponential decay. The first two solutions have been already found in
\cite{mancini}. We claim that these explicit solutions are natural candidates for the best constant in the Sobolev-type inequalities
$$\|u\|_{q}\le C\|\nabla u\|_2\,,\quad q=\frac{2n-1}{n-1}\,,\quad q=\frac{2n}{n-1}\,,\quad q=\frac{2n+2}{n-1}\,,
$$
see \cite{bhakta}. These inequalities can be obtained by interpolating between the L$^2$ gap inequality \eqref{poincare} and the Sobolev inequality \eqref{Sobeuc}.

\noindent $\bullet$
{\bf Asymptotic behavior in the linear case.} If $c=\Lambda$, then by \cite{bea, D} we know that \eqref{eigen} admits positive solutions $u$. Moreover, these
solutions satisfy the upper bound $u(r)\le C(1+r)e^{-\frac{n-1}2\,\,r}$, see e.g.\ \cite{APM}.
To show a similar lower bound we proceed using a strategy which inspires
the one in the proof of Theorem \ref{osc}. Any radial solution $u$ to \eqref{eigen} with $c=\Lambda$ and $u(0)>0$ satisfies the inequality
\[
0=u^{\prime\prime}(r)+(n-1)(\coth r)\, u^\prime(r)+\Lambda u(r)\le u^{\prime\prime}(r)+(n-1)u^\prime(r)+\Lambda u(r)
\]
as long as $u^\prime\le0$. In turn, this happens as long as $u\ge0$ since $[(\sinh r)^{n-1}u^\prime]^\prime=-\Lambda u$ and $u^\prime(0)=0$.
Therefore, as long as $u\ge0$, we have
\neweq{onemore}
\left[r^2\left(\frac{e^{\frac{n-1}2\,\,r}}ru(r)\right)^\prime\right]^\prime\ge0.
\endeq
Integrating this inequality on $[0,r]$ gives
\neweq{forzamilan}
\left(\frac{e^{\frac{n-1}2\,\,r}}{r}u(r)\right)^\prime\ge-\frac{u(0)}{r^2}.
\endeq
  Since the derivative of the function $r\mapsto e^{(n-1)r/2}u(r)$ is positive at $r=0$, we know that
$$\delta:=e^{(n-1)\varepsilon/2}u(\varepsilon)-u(0)>0$$
provided $\varepsilon>0$ is sufficiently small. Choose one such $\varepsilon$ and integrate \eqref{forzamilan} on $[\varepsilon,r]$ to get
\[
\frac{e^{\frac{n-1}2\,\,r}}ru(r)\ge \frac{e^{\frac{n-1}2\,\,\varepsilon}}\varepsilon u(\varepsilon)-u(0)\left(\frac1\varepsilon-\frac 1r\right)
\ge\frac{\delta}{\varepsilon}\qquad\forall r\ge\varepsilon.
\]
Hence, $u$ never vanishes and there exists $K>0$ such that $u(r)\ge K(1+r)e^{-\frac{n-1}2r}$ for all $r\ge0$.

If $0<c<\Lambda$, we claim that for all $\varepsilon>0$ and suitable constants $c_0,c_1(\varepsilon)>0$ we have
\begin{equation}\label{double.bounds}
 c_0\,e^{-\lambda_1r} \le u(r)\le c_1e^{-(\lambda_1-\varepsilon)r}.
\end{equation}
This means that radial solutions to \eqref{eigen} with $u(0)>0$ are positive and slowly decaying at infinity.
To prove this claim we set
\neweq{lambdac}
\lambda_1=\frac{n-1-\sqrt{(n-1)^2-4c}}{2}\, ,\qquad \lambda_2=\frac{n-1+\sqrt{(n-1)^2-4c}}{2}
\endeq
so that $\lambda_2>\lambda_1>0$. Then, by arguing as for \eq{onemore}, we see that in any interval $[0,R]$ on which $u$ is positive we get
\begin{equation}\label{lin.ge}
\left[e^{(\lambda_2-\lambda_1)r}\left(e^{\lambda_1r}u(r)\right)^\prime\right]^\prime\ge0.
\end{equation}
By integration we obtain $\left(e^{\lambda_1r}u(r)\right)^\prime\ge \lambda_1u(0)e^{(\lambda_1-\lambda_2)r}$. Integrating again yields:
\begin{equation}\label{hihi}
u(r)\ge \frac{u(0)}{\lambda_2-\lambda_1}\left[\lambda_2e^{-\lambda_1r}-\lambda_1e^{-\lambda_2r}\right].
\end{equation}
This first shows that $u$ never vanishes since the r.h.s.\ is always positive, and then that $u$ is lower bounded at infinity by a multiple of $e^{-\lambda_1r}$ with $\lambda_1$ given in \eq{lambdac}, so the lower bound \eqref{double.bounds} is proven.
As for the upper bound,  integrating \eqref{lin.ge} once and taking the conditions at $r=0$ into account gives, for all $r\ge0$,
\[
u^\prime(r)+\lambda_1u(r)\ge \lambda_1u(0)e^{-\lambda_2r}>0.
\]
By \eq{hihi} we know that $u$ is everywhere positive, hence we have that $\Theta(r)\ge -\lambda_1$ for all $r\ge0$,
where $\Theta(r):=u^\prime(r)/u(r)$. Recall also that $\Theta$ is negative. We now compute
\neweq{der}
\Theta^\prime(r)=\frac{u^{\prime\prime}(r)u(r)-u^\prime(r)^2}{u(r)^2}=\frac{u^{\prime\prime}(r)}{u(r)}-\Theta(r)^2=-(n-1)(\coth r)\Theta(r)-c-\Theta(r)^2.
\endeq
Assume first that $\Theta$ has infinitely many stationary points $r_m$, so that $\Theta^\prime(r_m)=0$ and $r_m\to+\infty$ as $m\to+\infty$. Hence, using \eqref{der}, we see that
\[
\Theta(r_m)\left[(n-1)(\coth r_m)+\Theta(r_m)\right]+c=0.
\]
This implies that, as $m\to+\infty$, $\Theta(r_m)$ tends either to $-\lambda_1$ or to $-\lambda_2$. Since $\Theta(r)\in[-\lambda_1,0]$, only the first possibility
can hold and $\Theta(r)\downarrow -\lambda_1$ as $r\to+\infty$. If instead $\Theta$ has finitely many or no stationary points, then $\Theta(r)$ has a limit as $r\to+\infty$ and \eq{der} shows that $\Theta^\prime(r)$ as well has a limit, necessarily zero (recall that $\Theta(r)\in[-\lambda_1,0]$). This corresponds to $(n-1)(\coth r)\Theta(r)+c+\Theta^2(r)\to0$ as $r\to+\infty$, which again implies that $\Theta(r)\downarrow -\lambda_1$ as $r\to+\infty$. Hence this latter fact holds true in any case. In particular for all $\varepsilon>0$ and all $r$ sufficiently large we have $u^\prime(r)/u(r)\le -(\lambda_1-\varepsilon)$, so that $u(r)\le c_1e^{-(\lambda_1-\varepsilon)r}$ for all $r\ge0$ and a suitable
$c_1(\varepsilon)>0$. This is precisely the upper bound in \eqref{double.bounds}.

Finally we discuss the case $c>\Lambda$, so that $c$ belongs to the L$^2$-spectrum. As a straightforward application of
\cite[Theorem 2.1 (b)]{DZ}, we see that solutions to \eq{eigen} change sign infinitely many times. It is also known \cite{APM} that any
eigenfunction $u$ satisfies $|u(r)|\le C(1+r)e^{-\frac{n-1}2\,\,r}$ where, by the above calculations, the r.h.s.\ describes the
asymptotic behavior of any eigenfunction corresponding to the eigenvalue $\Lambda$.\par\medskip

\noindent $\bullet$ {\bf The two-dimensional case.} Since when $n=2$ any exponent $p>1$ is subcritical, Theorems \ref{polynomial} and \ref{osc} hold for any $p>1$ when $n=2$. Moreover, when $n=2$, Theorem \ref{sublinear} holds with no modifications.\par

\medskip

\noindent{\bf $\bullet $ Numerics and open problems}. 1) Is it possible to get explicit bounds on $U(0)$, as given by Theorem \ref{polynomial}?

\noindent 2) In the subcritical case, numerical analysis seems to show that the number of zeros of a sign-changing solution increases as $u(0)$ increases, see
Figures 5-6. Several natural questions then arise. Is it true that the number of zeros of $u$ in nondecreasing as $u(0)$ increases? What is the asymptotic
behavior as $k\to\infty$ of the shooting levels $\alpha_k$ where the solution to \eqref{equazione}--\eqref{initcond} switches from $k$ to $k+1$
zeros? We conjecture that the solutions corresponding to $\alpha_k$ have finite energy. In this respect, \cite[Theorem 5.1]{bhakta} proves the
existence of infinitely many radial sign-changing finite energy solutions to \eqref{triangle}, whose energy is arbitrarily large.

\noindent 3) Numerics shows that in the supercritical case and for large dimensions and $p$ large the solutions are ordered and do not intersect. The corresponding result is well-known and quite interesting in the Euclidean setting, see \cite{JoLu}. This seems to require further investigation.


\section{\bf A Poho\v{z}aev-type functional}\label{pohofunctional}

We are here interested in studying the behavior of local solutions to \eqref{equazione}--\eqref{initcond}, namely solutions to the Cauchy problem
\neweq{cauchy}
\left\{\begin{array}{lll}
u''(r)+(n-1)(\coth r)\,u'(r)+|u(r)|^{p-1}u(r)=0\qquad (r>0)\\
u(0)=\alpha\,,\qquad u'(0)=0
\end{array}\right.
\endeq
for some $\alpha>0$. By arguing as in Proposition 1 in the Appendix of \cite{NS2}, one sees that \eqref{cauchy} has a $C^2$ local solution.
In fact, these solutions are global and vanish at infinity. This fact is known in the case $1<p<\frac{n+2}{n-2}$ from \cite[Theorem 4.1]{bhakta} in case of general solutions to \eqref{triangle}. We give here a simpler proof in the case of radial solutions which works for any $p>0$.

\begin{lemma}\label{zero}
Let $p>0$. For any $\alpha>0$ the local solution $u=u(r)$ to \eqref{cauchy} may be continued for all $r>0$ and $\lim_{r\to+\infty}u(r)=0$. Also in the non-Lipschitz case $p\in(0,1)$ each solution intersects the $r$-axis transversally.
\end{lemma}
\begin{proof} We introduce the Lyapunov functional
\neweq{F}
F(r):=\frac 12u^\prime(r)^2+\frac 1{p+1}|u(r)|^{p+1}
\endeq
and we show that $F$ is decreasing. Indeed, by \eqref{equazione} we get
\beq\label{Fprime}
F^\prime(r)=\Big[u^{\prime\prime}(r)+|u(r)|^{p-1}u(r)\Big]u^\prime(r)=-(n-1)(\coth r)u^\prime(r)^2\le0.
\eeq
This implies in particular that both $u$ and $u^\prime$ are bounded. A straightforward calculation using \eqref{Fprime} shows that
\[
\left[(\sinh r)^{2(n-1)}F(r)\right]^\prime\ge0
\]
with strict inequality holding at least for $r$ small. Hence $F(r)>0$ for all $r\ge0$ and in particular $u$ may intersect the $r$-axis only transversally.

Suppose that $u$ does not admit a limit as $r\to+\infty$.
Since $F$ is decreasing, $u$ cannot oscillate while having a constant sign. Then $u$ admits infinitely many negative minima and infinitely
many positive maxima. Let $r_1^{(k)}$ be the sequence of zeros of $u$ and $r_2^{(k)}$ be the sequence of the first maximum points of $u$
after $r_1^{(k)}$. We then have:
\beq\label{energia}\begin{aligned}
|F(r_2^{(k)})-F(r_1^{(k)})|&=\left|\int_{r_1^{(k)}}^{r_2^{(k)}}F^\prime(r)\,{\rm d}r\right|=(n-1)\int_{r_1^{(k)}}^{r_2^{(k)}}
(\coth r)\,u^\prime(r)^2\,{\rm d}r\\
&>(n-1)\int_{r_1^{(k)}}^{r_2^{(k)}}u^\prime(r)^2\,{\rm d}r=-(n-1)\int_{r_1^{(k)}}^{r_2^{(k)}}u(r)u^{\prime\prime}(r)\,{\rm d}r\\
&=(n-1)\int_{r_1^{(k)}}^{r_2^{(k)}}u(r)\left[(n-1)(\coth r)\,u^\prime(r)+u(r)^p\right]\,{\rm d}r\\
&> (n-1)^2\int_{r_1^{(k)}}^{r_2^{(k)}}u(r)\,u^\prime(r)\,{\rm d}r=\frac{(n-1)^2}2u(r_2^{(k)})^2,
\end{aligned}\eeq
where we used the positivity of $u$ and $u^{\prime}$ in $[r_1^{(k)},r_2^{(k)}]$ and the fact that $u\left(r_1^{(k)}\right)=u^\prime\left(r_2^{(k)}\right)=0$ for all $k$. As $F(r)$ is decreasing, it has a finite nonnegative limit.
Therefore, the l.h.s.\ of \eqref{energia} tends to zero as $k\to+\infty$. We conclude that $u\left(r_2^{(k)}\right)$ tends to zero as well when $k\to+\infty$. Therefore $u$ tends to zero on its maxima. Similar considerations hold for the minima, so that we can conclude that $u(r)\to0$ as $r\to+\infty$ as claimed.\end{proof}

We now introduce the function
\neweq{phin}
\phi_n(r)=\int_0^r(\sinh s)^{n-1}\, {\rm d}s.
\endeq
We point out that the integral involving $\phi_n$ can be explicitly computed in terms of elementary functions, but since its form appears
complicated, we leave it as in \eqref{phin}. For any $p>0$ we also define the function
\neweq{psip}
\psi_p(r)=\frac{p+3}{2(p+1)}\, (\sinh r)^{n-1}-(n-1)\phi_n(r)\coth r\qquad (r>0)
\endeq
which is linked to local solutions to \eqref{equazione} by means of the following statement.

\begin{lemma}\label{local}
Let $p>0$. For any local solution $u=u(r)$ to \eqref{cauchy} let
\neweq{PSI}
\Psi(r):=\phi_n(r)\left(\frac{u'(r)^2}{2}+\frac{|u(r)|^{p+1}}{p+1}\right)+(\sinh r)^{n-1}\frac{u(r)u'(r)}{p+1}.
\endeq
Then
\neweq{Psiproperties}
\Psi(0)=0\qquad\mbox{and}\qquad\Psi'(r)=u'(r)^2\psi_p(r).
\endeq
\end{lemma}
\begin{proof} We use twice \eqref{equazione} to obtain
\[\begin{split}
\Psi^\prime(r)
&=(\sinh r)^{n-1}\left(\frac{u'(r)^2}{2}+\frac{|u(r)|^{p+1}}{p+1}\right)+\phi_n(r)\Big(u''(r)+|u(r)|^{p-1}u(r)\Big)u'(r)\\
&\,\,\,\,\, +\frac{n-1}{p+1}\, (\sinh r)^{n-2}(\cosh r)\ u(r)u'(r)+\frac{(\sinh r)^{n-1}}{p+1}\, u'(r)^2+
\frac{(\sinh r)^{n-1}}{p+1}\, u(r)u''(r)\\
&=(\sinh r)^{n-1}\left(\frac{u'(r)^2}{2}+\frac{|u(r)|^{p+1}}{p+1}\right)-(n-1)\phi_n(r)(\coth r)\, u'(r)^2\\
&\,\,\,\,\, +\frac{(\sinh r)^{n-1}}{p+1}\, u'(r)^2-\frac{(\sinh r)^{n-1}}{p+1}\, |u(r)|^{p+1}\\
&=\left(\frac{(\sinh r)^{n-1}}{2}-(n-1)\phi_n(r)(\coth r)+\frac{(\sinh r)^{n-1}}{p+1}\right)\, u'(r)^2.\end{split}\]
Recalling \eqref{psip}, this proves the statement.\end{proof}

In the supercritical case, the function $\Psi$ is negative:

\begin{lemma}\label{psipsi}
Assume $n\ge3$ and $p\ge\frac{n+2}{n-2}$. For any local solution $u=u(r)$ to \eqref{cauchy} the function
$\Psi$ defined in \eqref{PSI} satisfies $\Psi'(r)<0$ for all $r>0$. Hence, $\Psi(r)<0$ for all $r>0$.
\end{lemma}
\begin{proof} In view of \eqref{Psiproperties}, the statement follows if we show that $\psi_p(r)<0$ for all $r>0$. In turn,
since $\lim_{r\to0}\psi_p(r)=0$, it suffices to prove that $\psi_p'(r)<0$ for all $r>0$.\par
Some computations show that
\neweq{hh}
\psi_p'(r)=\frac{n-1}{(\sinh r)^2}\ \left[\phi_n(r)-\frac{p-1}{2(p+1)}\, (\sinh r)^n\, \cosh r\right]=:
\frac{n-1}{(\sinh r)^2}\, h(r).
\endeq
We are so led to determine the sign of $h$. Further computations show that
\neweq{hprime}
h'(r)=\frac{3p+1}{2(p+1)}\, (\sinh r)^{n-1}\, \left[1-\left(1+\frac{(n-2)p-(n+2)}{3p+1}\right)(\cosh r)^2\right].
\endeq
In view of the assumption $p\ge\frac{n+2}{n-2}$, this shows that $h'(r)<0$ for $r>0$.
Since $h(0)=0$, this shows that $h(r)<0$ for $r>0$ so that $\psi_p'(r)<0$ for all $r>0$.\end{proof}

In the subcritical case, we obtain a different statement

\begin{lemma}\label{subsub}
Assume $n\ge3$ and $1<p<\frac{n+2}{n-2}$. For any local solution $u=u(r)$ to \eqref{cauchy} the function $\Psi$
defined in \eqref{PSI} admits a limit as $r\to+\infty$.
\end{lemma}
\begin{proof} Let $\psi_p$ be as in \eqref{psip}. Since $1<p<\frac{n+2}{n-2}$, one sees that
\neweq{Rnp}
\exists! R_{n,p}>0\qquad\mbox{such that}\quad\psi_p(R_{n,p})=0
\endeq
and $\psi_p(r)>0$ for $r<R_{n,p}$ whereas $\psi_p(r)<0$ for $r>R_{n,p}$. Then \eqref{Psiproperties} shows that $\Psi'(r)<0$ for all
$r>R_{n,p}$ so that $r\mapsto\Psi(r)$ is eventually decreasing and admits a limit as $r\to+\infty$.\end{proof}

\begin{remark}\label{psisub}
{\em In the sublinear case $0<p\le1$, \eqref{hh} yields $\psi'_p(r)>0$ for $r>0$. Hence, $\psi_p(r)>0$ and, by \eqref{Psiproperties},
one sees that $\Psi'(r)>0$ for all $r>0$. Finally, $\Psi(r)>0$ and $\Psi$ admits a positive limit (possibly $+\infty$)
as $r\to+\infty$.}\end{remark}

\section{\bf Supercritical case: proof of Theorem \ref{supercritico}}\label{proofsuper}
We treat the case $p\ge (n+2)/(n-2)$. It is clearly sufficient to deal with radial solutions satisfying $u(0)>0$.
Notice that \eqref{equazione} may be rewritten as
\neweq{alternative}
\frac{1}{(\sinh r)^{n-1}}\, \Big[(\sinh r)^{n-1}\, u'(r)\Big]'=-|u(r)|^{p-1}u(r)
\endeq
so that the map $r\mapsto(\sinh r)^{n-1}\, u'(r)$ is strictly decreasing as long as $u(r)$ remains positive.
Since its value at $r=0$ is 0, we infer that $u'(r)<0$ as long as $u(r)$ remains positive and two cases may occur:
\begin{enumerate}
\item There exists $\rho>0$ such that $u(r)>0$ for $r\in[0,\rho)$, $u(\rho)=0$, and $u'(\rho)<0$;
\item $u(r)>0$ for all $r\in[0,\infty)$.
\end{enumerate}

If $\rho>0$ as in case (1) exists, then $\Psi(\rho)>0$, where $\Psi$ in as \eqref{PSI}, contradicting Lemma \ref{psipsi}.
This rules out case (1) and shows that case (2) occurs.\par

This, together with Lemma \ref{zero}, shows that for any $\alpha>0$ the Cauchy problem \eqref{cauchy} admits a unique global positive
solution which vanishes at infinity. This proves the first assertion in Theorem \ref{supercritico}.

The rest of this section is devoted to the proof of the limit \eqref{ratesuper}. This requires several intermediate results.

\begin{lemma}\label{lowerbound} Assume that $p\ge\frac{n+2}{n-2}$ and let $u=u(r)$ be a solution to \eqref{cauchy}.
Then there exist $C_0>0$ and $r_0>0$ such that
$$u(r)>C_0\ e^{-\frac{n-1}{2}r}\qquad\mbox{for all }r>r_0.$$
\end{lemma}
\proof We use the function $\Psi$ defined in \eqref{PSI}. By Lemma \ref{psipsi} we see that there exists $\delta>0$ and $r_0>0$ such that
$$
\Psi(r)<-\delta\qquad\mbox{for all }r>r_0.
$$
Therefore,
$$
(\sinh r)^{n-1}\, \frac{u(r)u'(r)}{p+1}<-\delta\qquad\mbox{for all }r>r_0
$$
so that, for a suitable constant $C>0$,
$$
u(r)u'(r)<-C e^{(1-n)r}\qquad\mbox{for all }r>r_0\ .
$$
By integrating this inequality over $(r,+\infty)$ we get
$$
-\frac{u(r)^2}{2}<-\frac{C}{n-1} e^{(1-n)r}\qquad\mbox{for all }r>r_0
$$
and the stated lower bound follows.\endproof

\begin{lemma}\label{exp} There exist no strictly positive constants $C,\beta$ such that the bound $u(r)\le Ce^{-\beta r}$ holds for all $r\ge0$.
\end{lemma}
\proof In the sequel, $C$ will denote a positive constant which can change from line to line. Suppose by contradiction that,
for suitable $C,\beta>0$ the bound
\begin{equation}\label{upper}
u(r)\le Ce^{-\beta r}\ \end{equation}
is satisfied for all $r\ge0$. Using \eqref{alternative} we get
\[
\frac1{(\sinh r)^{n-1}}\left[(\sinh r)^{n-1}u^\prime(r)\right]^\prime=-u(r)^p\ge -Ce^{-p\beta r}
\]
and hence, using the inequality $\sinh r\le e^r/2$,
\begin{equation}\label{lower}
\left[(\sinh r)^{n-1}u^\prime(r)\right]^\prime\ge -Ce^{(n-1-p\beta)r}.
\end{equation}
In order to reach a contradiction one can assume that $\alpha$ is small, in particular that $\beta <(n-1)/p$.
Let us assume this condition and integrate \eqref{lower} between 0 and $r$. We get, for all $r>0$:
\[
(\sinh r)^{n-1}u^\prime(r)\ge C\left(1-e^{(n-1-p\beta)r}\right).
\]
Therefore, for $r$ sufficiently large, recalling that $n-1-p\beta>0$:
\[
u^\prime(r)\ge Ce^{-(n-1)r}\left(1-e^{(n-1-p\beta)r}\right)\ge -Ce^{-p\beta r}.
\]
The latter inequality can be integrated between $r$ and $+\infty$ so that, recalling that $u(r)\to 0$ as $r\to+\infty$, we have $-u(r)\ge-Ce^{-p\beta r}$, or:
\[
u(r)\le Ce^{-p\beta r},
\]
first for sufficiently large $r$ and then for all $r$ since $u$ is continuous. Therefore, compared with \eqref{upper}, we have proven a faster decay at infinity for $u$, since $p>1$. The procedure can be iterated to prove that $u(r)\le Ce^{-p^k\beta r}$ for all positive integers $k$ such that $p^{k-1}\beta<(n-1)/p$ and for all
$r\ge0$. Therefore we conclude that the bound $u(r)\le C_\varepsilon e^{-(n-1-\varepsilon)r}$ holds, given any positive $\varepsilon$, for a suitable constant $C_\varepsilon$ and for all positive $r$. This contradicts Lemma \ref{lowerbound}. \endproof

\begin{lemma}\label{theta}
The equality
\[
\lim_{r\to+\infty}\frac{u^\prime(r)}{u(r)}=0
\]
holds true.
\end{lemma}
\proof We first derive an auxiliary inequality. Let $\phi_n$ be defined as in \eqref{phin}. We notice that by de l'H\^{o}pital's rule
\neweq{triviallimit}
\lim_{r\to\infty}\ \frac{\phi_n(r)}{(\sinh r)^{n-1}}\ =\ \lim_{r\to\infty}\ \frac{(\sinh r)^{n-1}}
{(n-1)(\sinh r)^{n-2}(\cosh r)}\ =\ \frac{1}{n-1}.
\endeq
This proves that for any $\eps>0$ there exists $r_\eps>0$ such that
\neweq{reps}
\frac{\phi_n(r)}{(\sinh r)^{n-1}}>\frac{1}{n-1+\eps}\qquad\mbox{for all }r>r_\eps.
\endeq
{From} Lemma \ref{psipsi} we know that the function $\Psi$ defined in \eqref{PSI} is strictly negative, namely
$$
\frac{\phi_n(r)}{(\sinh r)^{n-1}}\, \left(\frac{u'(r)^2}{2}+\frac{u(r)^{p+1}}{p+1}\right)+
\frac{u(r)u'(r)}{p+1}<0\qquad\mbox{for all }r>0.
$$
In turn, by using \eqref{reps} this shows that
$$
\frac{1}{n-1+\eps}\, \frac{u'(r)^2}{2}\ +\ \frac{u(r)u'(r)}{p+1}<0\qquad\mbox{for all }r>r_\eps.
$$
Dividing by $u'(r)<0$ the latter inequality we obtain
\begin{equation}\label{second lower}u'(r)+\frac{2(n-1+\eps)}{p+1}\, u(r)>0\qquad\mbox{for all }r>r_\eps.\end{equation}

Set now $\Theta(r):=u^\prime(r)/u(r)$, and notice that $\Theta(0)=0$, $\Theta(r)<0$ for all $r>0$. We can rule out the possibility that
$$\limsup_{r\to+\infty}\Theta(r)<0$$
since this would imply a uniform exponential upper bound for $u$, against Lemma \ref{exp}. Therefore, either the statement is true or $\Theta$ has no limit. Should the latter possibility hold, $\Theta$ has infinitely many maxima and minima $r_m$, at which $\Theta^\prime$ vanishes, with $r_m\to+\infty$ as $m\to+\infty$. This means that, at each $r_m$, $u^{\prime\prime}u-(u^\prime)^2=0$. Multiplying equation \eqref{equazione} by $u$, we thus get, for all $m$,
\[
u^\prime(r_m)\left[(n-1)u(r_m)(\coth r_m)+u^\prime(r_m)\right]=-u(r_m)^{p+1}.
\]
Recalling that $u$ never vanishes shows that the quantity in the l.h.s.\ above never vanishes as well, so that we may rewrite the above equality as
\begin{equation}\label{stationarypoints}
u^\prime(r_m)=-\frac{u(r_m)^{p+1}}{(n-1)u(r_m)(\coth r_m)+u^\prime(r_m)}.
\end{equation}

Since $p>1$, using the auxiliary inequality \eqref{second lower}, we have that
\[\begin{aligned}
(n-1)u(r_m)(\coth r_m)+u^\prime(r_m)&>u(r_m)\left[(n-1)(\coth r_m)-2\frac{n-1+\eps}{p+1}\right]\\
&=u(r_m)(n-1)\left[(\coth r_m)-\frac 2{p+1}\, \frac{n-1+\eps}{n-1}\right]\\
&\ge \frac {u(r_m)}\alpha
\end{aligned}
\]
for a suitable positive constant $\alpha$, provided $0<\eps<(p-1)(n-1)/2$ and $m$ is sufficiently large, say $m\ge\overline{m}$. Putting this bound into
\eqref{stationarypoints} yields, for $m\ge \overline{m}$:
\[
0>u^\prime(r_m)>-\alpha u(r_m)^p
\]
which shows that, for the same set of indices,
\[
0>\frac{u^\prime(r_m)}{u(r_m)}>-\alpha u(r_m)^{p-1}.
\]
This and Lemma \ref{zero} imply that $u^\prime(r_m)/u(r_m)\to0$ as $m\to+\infty$. The definition of the sequence $\{r_m\}$ then shows that $u^\prime(r)/u(r)\to0$ as $r\to+\infty$ even if $\Theta$ has infinitely many stationary points. This concludes the proof.
\endproof
\begin{figure}[ht]
\centering
\includegraphics[height=6cm, width=14cm]{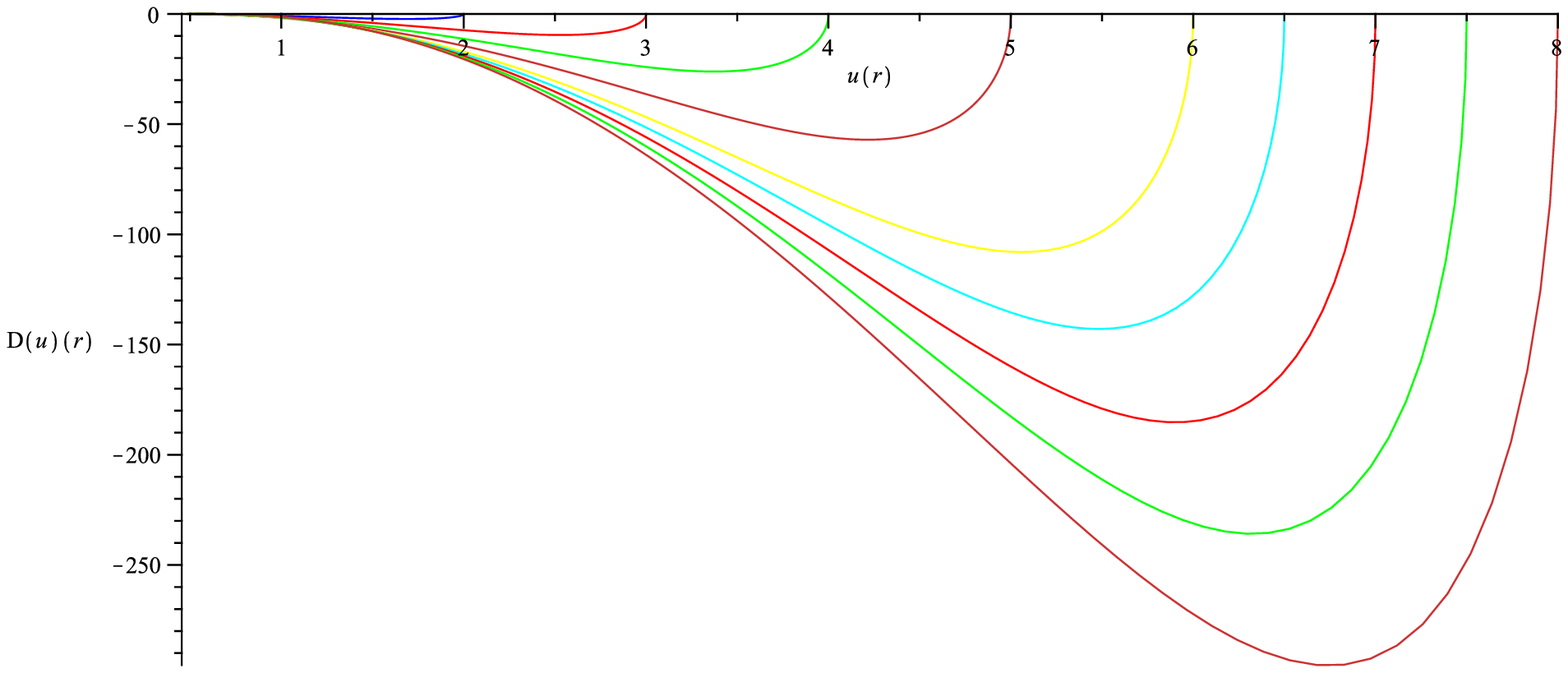}
\flushleft \color{darkred}\noindent\textit{Figure 1: phase plot of some solutions when $d=3$, $p=6$ (supercritical case).}
\end{figure}
We have now all the ingredients to prove \eqref{ratesuper}. By using an argument similar to the one in the proof of Lemma \ref{theta} we show that $\Theta_p(r):=u^\prime(r)/u(r)^p$ has a limit as $r\to+\infty$. This would clearly hold if $\Theta_p$ has finitely many maxima and minima (or none at all). If there is instead a sequence $\{r_m\}$ of extremals, with $r_m\to+\infty$ as $m\to+\infty$, we would have, by computing explicitly derivatives and recalling that $u$ never vanishes, $u^{\prime\prime}(r_m)u(r_m)-pu^\prime(r_m)^2=0$. Using \eqref{equazione} this would imply that
\[
u^\prime(r_m)=-\frac{u(r_m)^{p+1}}{(n-1)u(r_m)(\coth r_m)+pu^\prime(r_m)}.
\]
But we have proved in Lemma \ref{theta} that $u^\prime=o(u)$ as $r\to+\infty$, in particular along the sequence $\{r_m\}$. Therefore we have
$u^\prime(r_m)\sim-u(r_m)^p/(n-1)$ (where $f\sim g$ means that $f/g\to1$), and in particular $\Theta_p(r_m)\to-1/(n-1)$ as $m\to+\infty$. By the definition of the sequence $\{r_m\}$ this entails that $\Theta_p(r)\to-1/(n-1)$ as $r\to+\infty$. In any case $\Theta_p(r)$ has a limit as $r\to+\infty$.

Using again \eqref{equazione}, we may write
\[
\frac{u^{\prime\prime}(r)}{u^\prime(r)}+\frac{u(r)^p}{u^\prime(r)}\to 1-n\ \ {\rm as}\ r\to+\infty.
\]
Since $\frac{u^p(r)}{u^\prime(r)}$ has been shown to have a limit as $r\to+\infty$, $\frac{u^{\prime\prime}(r)}{u^\prime(r)}$ has a limit as well. But then de l'H\^{o}pital's rule and Lemma \ref{theta} imply that
\[
\lim_{r\to+\infty}\frac{u^{\prime\prime}(r)}{u^\prime(r)}=\lim_{r\to+\infty}\frac{u^\prime(r)}{u(r)}=0,
\]
as claimed in the statement. Thus we have proved that
\[
\lim_{r\to+\infty}\frac{u^\prime(r)}{u(r)^p}=\frac1{1-n}.
\]
\begin{figure}[ht]
\centering
\includegraphics[height=6cm, width=14cm]{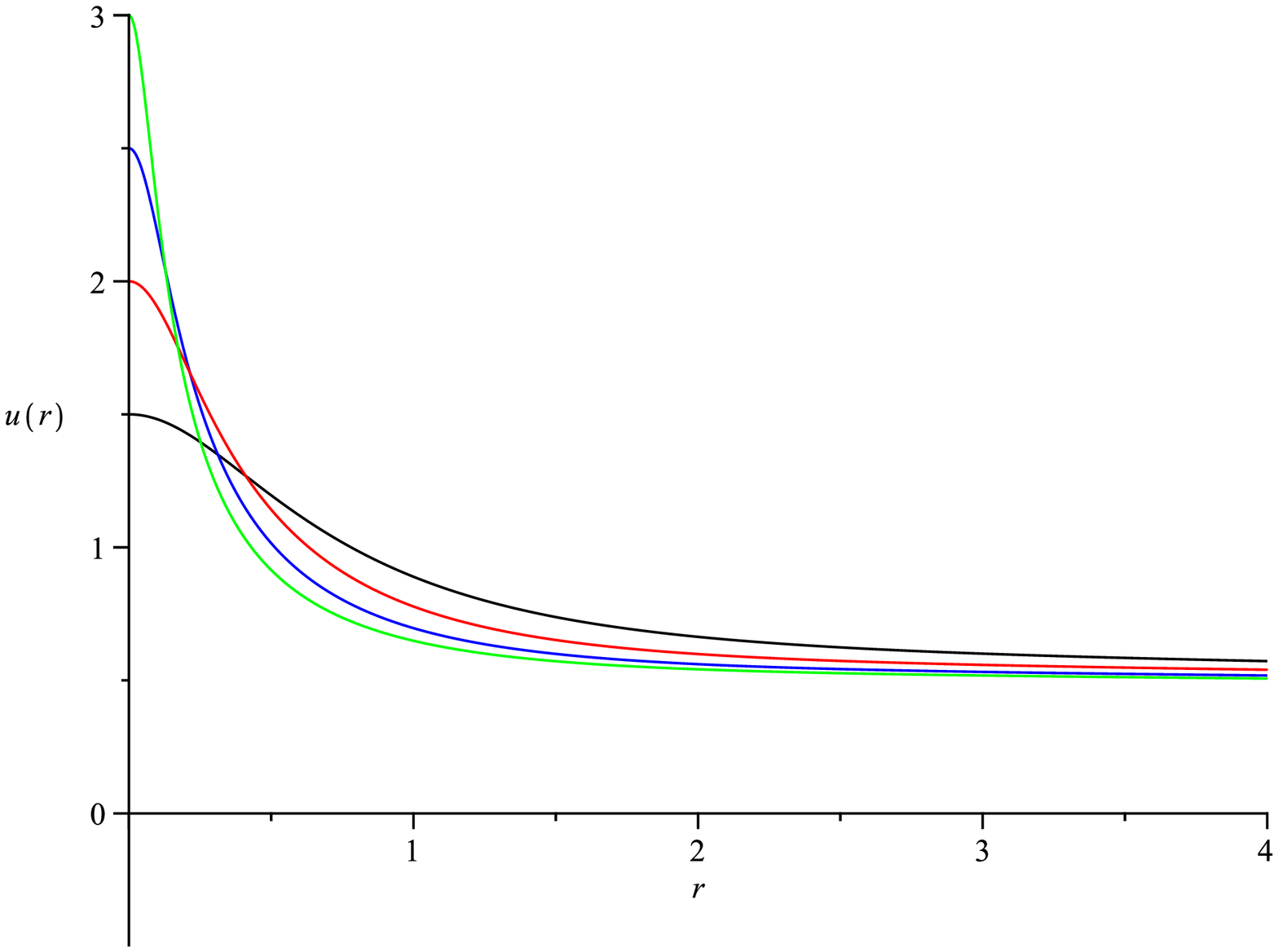}
\flushleft \color{darkred}\noindent\textit{Figure 2: plot of some solutions when $d=3$, $p=6$ (supercritical case).}
\end{figure}
Therefore, for all $\varepsilon>0$ there exists $r_\varepsilon$ such that, for all $r\ge r_\varepsilon$:
\[
\frac{p-1}{n-1}-\varepsilon\le \left(u^{1-p}\right)^\prime(r)\le\frac{p-1}{n-1}+\varepsilon.
\]
By integration between $r_\varepsilon$ and $r$ we then get
\[\begin{split}
\left[\left(\frac{p-1}{n-1}+\varepsilon\right)(r-r_\varepsilon)+u(r_{\varepsilon})^{1-p}\right]^{-1/(p-1)}&\le u(r)\\
&\le\left[\left(\frac{p-1}{n-1}-\varepsilon\right)(r-r_\varepsilon)+u(r_{\varepsilon})^{1-p}\right]^{-1/(p-1)}.
\end{split}
\]
Multiplying such inequalities by $r^{1/(p-1)}$ we obtain
\begin{eqnarray*}
\left(\frac{p-1}{n-1}+\varepsilon\right)^{-1/(p-1)} &\le& \liminf_{r\to+\infty}r^{1/(p-1)}u(r)\\
&\le& \limsup_{r\to+\infty}r^{1/(p-1)}u(r)\le \left(\frac{p-1}{n-1}-\varepsilon\right)^{-1/(p-1)}.
\end{eqnarray*}
Finally, since this holds for all positive $\varepsilon$, the proof of \eqref{ratesuper} is complete.\endproof

\section{\bf Subcritical case: proof of Theorem \ref{polynomial}}\label{proofsub}

Let $u$ be a local solution to \eqref{cauchy} satisfying \eqref{shooting}. By \cite[Corollary 4.6]{mancini}, there exists a unique $r_0>0$ such
that $u(r_0)=U(r_0)$. Therefore, $u$ is positive on $[0,+\infty)$. Moreover, we have

\begin{lemma}\label{upperbound} Assume that $p>1$, let $u$ be a solution to \eqref{cauchy} which is positive for all $r>0$ and assume that
there exist $C,\alpha>0$ such that $u(r)\le Ce^{-\alpha r}$ for all $r>0$. Then also the bound
\neweq{goodecay}
u(r)\le Ae^{-(n-1)r}
\endeq
holds for a suitable positive constant $A$ and for all $r>0$.
\end{lemma}
\begin{proof} Proceeding as in the proof of Lemma \ref{exp}, we can improve the bound $u(r)\le Ce^{-\alpha r}$ by showing that
$u(r)\le C_\varepsilon e^{-(n-1-\varepsilon)r}$ for any $\varepsilon>0$. To arrive at the stated upper bound, we go back to the proof of that Lemma.
It has been shown there that the bound $u(r)\le Ce^{-\alpha r}$ implies that $\left[(\sinh r)^{n-1}u^\prime(r)\right]^\prime\ge -Ce^{(n-1-p\alpha)r}$,
with no restriction on $\alpha$ and $p$ needed up to that point. We can then take $\alpha$ sufficiently close to $n-1$ so that $n-1-p\alpha<0$.
Integrating the latter differential inequality between 0 and $r$ we get, say for all $r\ge1$, $(\sinh r)^{n-1}u^\prime(r)\ge -C$ for a suitable $C>0$.
We integrate again between $r$ and $+\infty$ so that, recalling that $\lim_{r\to+\infty}u(r)=0$, we have $-u(r)\ge-Ae^{-(n-1)r}$, or $u(r)\le Ae^{-(n-1)r}$.\end{proof}
\begin{figure}
\centering
\includegraphics[height=6cm, width=12cm]{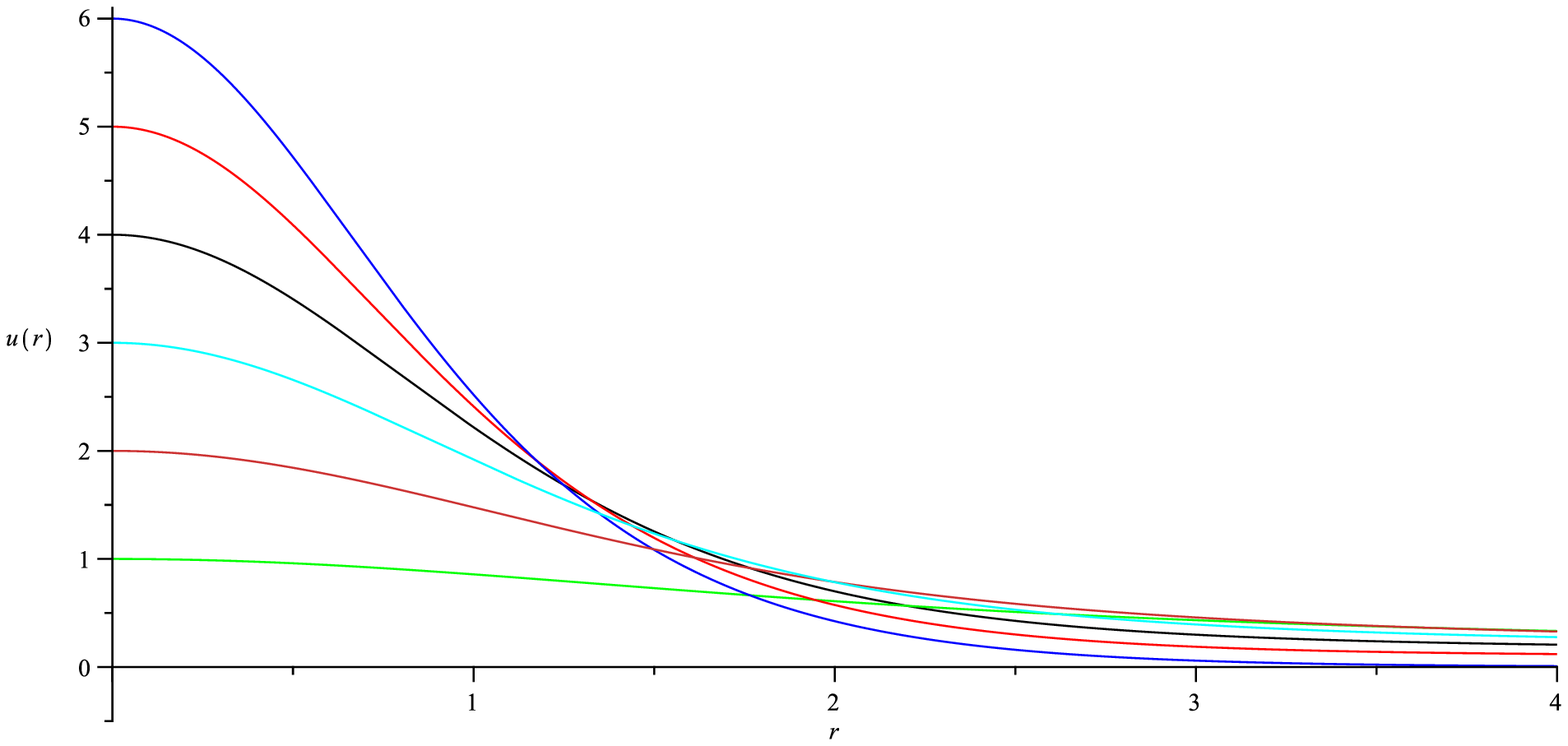}
\flushleft \color{darkred}\noindent\textit{Figure 3: plot of some positive solutions when $d=3$, $p=2$ (subcritical case). The special exponentially decaying solution $U$ corresponds to the blue line ($U(0)=6$)}
\end{figure}

By Theorem A we know that any $u$ as in the statement of Theorem \ref{polynomial} cannot satisfy \eqref{goodecay}. In fact we shall now prove that  \eqref{goodecay} improves to \eqref{n-1}. Notice indeed that from equation \eqref{equazione} we learn that $r\mapsto(\sinh r)^{n-1}U^\prime(r)$ is decreasing and admits a limit $\ell\in[-\infty,0)$. If $\ell=-\infty$,
de l'H\^{o}pital's rule shows that $U(r)e^{(n-1)r}\to+\infty$ as $r\to+\infty$, in contradiction with \eqref{goodecay}. Therefore, there exists $\gamma>0$ such that
\neweq{claimderivative}
\lim_{r\to+\infty} e^{(n-1)r}U^\prime(r)=-\gamma.
\endeq
Using again de l'H\^{o}pital's rule yields that
\neweq{exactbound}
\lim_{r\to+\infty} e^{(n-1)r}u(r)=\frac{\gamma}{n-1}
\endeq
as claimed. Hence, Lemma \ref{upperbound} shows that also
Lemma \ref{exp} holds.\par
According to Lemma \ref{subsub}, two cases may occur:
$$(i)\lim_{r\to+\infty}\Psi(r)<0\qquad(ii)\lim_{r\to+\infty}\Psi(r)\ge0.$$
If case $(i)$ occurs, then we obtain again \eqref{second lower} provided $r_\eps$ is sufficiently large.
Since Lemma \ref{exp} and \eqref{second lower} hold,
we can proceed exactly as in the supercritical case $p\ge\frac{n+2}{n-2}$, see all what follows \eqref{second lower}, and obtain \eqref{pol}.\par
If case $(ii)$ occurs, then $\Psi(r)>0$ for all $r>0$, that is
$$\phi_n(r)\left(\frac{u'(r)^2}{2}+\frac{u(r)^{p+1}}{p+1}\right)+(\sinh r)^{n-1}\frac{u(r)u'(r)}{p+1}>0\qquad\mbox{for all }r>0.$$
In turn, since $\phi_n(r)<(\sinh r)^{n-1}/(n-1)$, we obtain
$$u'(r)^2+\frac{2(n-1)}{p+1}\, u(r)u'(r)+\frac{2}{p+1}\, u(r)^{p+1}>0\qquad\mbox{for all }r>0.$$
We solve this as a second order inequality with respect to $u'(r)$. By Lemma \ref{zero} we see that the discriminant of this equation,
namely
$$\frac{(n-1)^2}{(p+1)^2}\, u(r)^2-\frac{2}{p+1}\, u(r)^{p+1}\, ,$$
is eventually positive, say for $r>r_0$ suitably large. By continuity, one of the following alternatives holds:
$$(a)\ u'(r)<-\frac{n-1}{p+1}\, u(r)-\left(\frac{(n-1)^2}{(p+1)^2}\, u(r)^2-\frac{2}{p+1}\, u(r)^{p+1}\right)^{1/2}\qquad\mbox{for all }r>r_0,$$
$$(b)\ u'(r)>-\frac{n-1}{p+1}\, u(r)+\left(\frac{(n-1)^2}{(p+1)^2}\, u(r)^2-\frac{2}{p+1}\, u(r)^{p+1}\right)^{1/2}\qquad\mbox{for all }r>r_0.$$
If case $(a)$ holds, then
$$\frac{u'(r)}{u(r)}<-\frac{n-1}{p+1}\qquad\mbox{for all }r>r_0.$$
By integration over $(r_0,r)$ we obtain
$$\log\frac{u(r)}{u(r_0)}\le-\frac{n-1}{p+1}\, (r-r_0)\qquad\mbox{for all }r>r_0$$
so that $u(r)\le ce^{-\frac{n-1}{p+1}r}$, in contradiction with Lemma \ref{exp}.\par
\begin{figure}
\centering
\includegraphics[height=6cm, width=12cm]{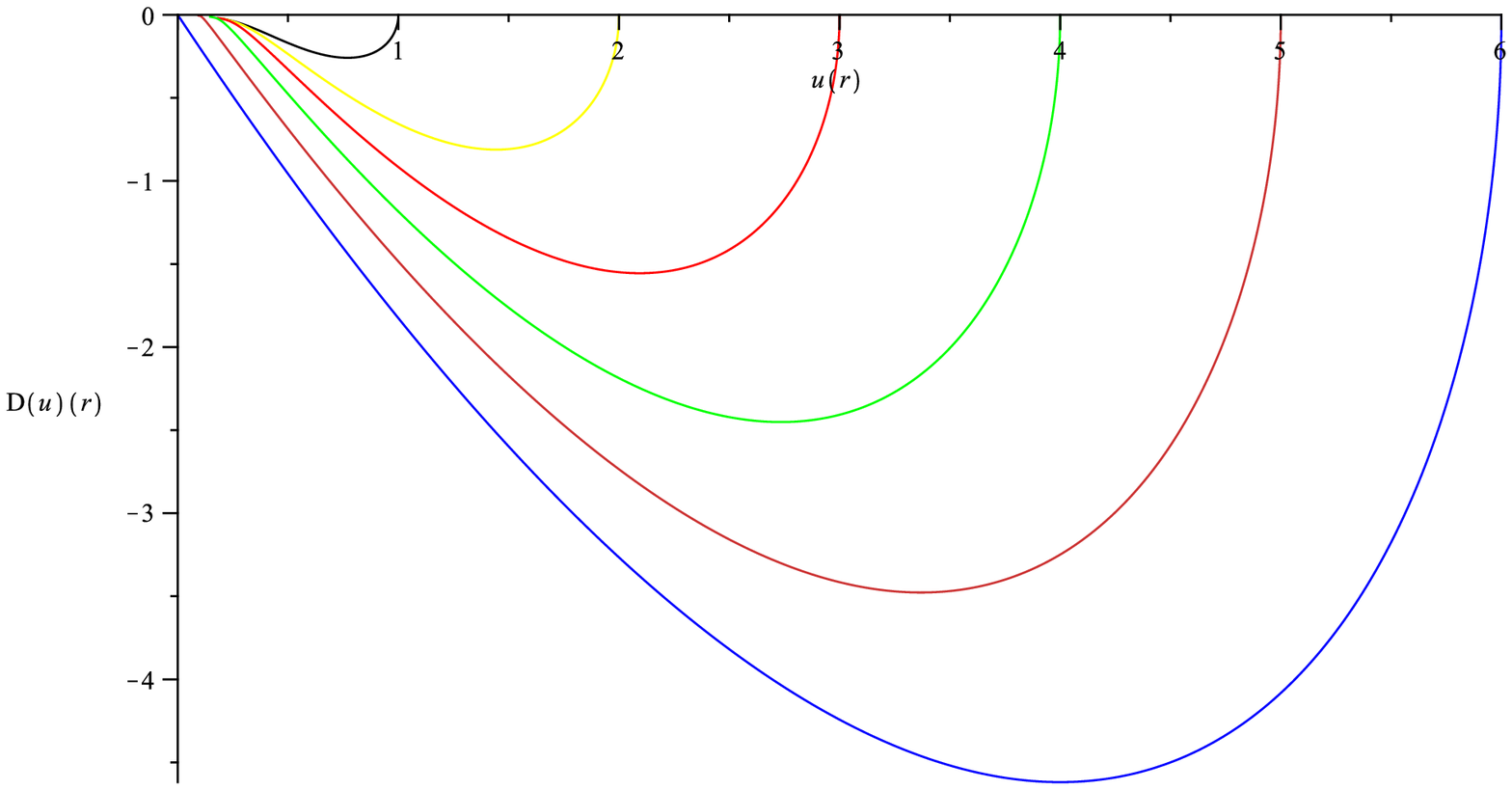}
\flushleft \color{darkred}\noindent\textit{Figure 4: phase plot of some positive solutions when $d=3$, $p=2$ (subcritical case). The special exponentially decaying solution $U$ corresponds to the blue line ($U(0)=6$)}
\end{figure}
Therefore, case $(b)$ holds, namely
$$u'(r)>\frac{n-1}{p+1}\, u(r)\, \left[-1+\left(1-\frac{2(p+1)}{(n-1)^2}\, u(r)^{p-1}\right)^{1/2}\right]\qquad\mbox{for all }r>r_0.$$
We may then exploit the inequality $\sqrt{1-\alpha}\ge1-\alpha$ valid for all $\alpha\in[0,1]$ to get
$$u'(r)>-\frac{2}{n-1}\, u(r)^p\qquad\mbox{for all }r>r_0$$

where $r_0\ge0$ is such that $u(r)<1$ for all $r>r_0$. By integrating this inequality over $(r_0,r)$ we obtain
$$u(r)>\frac{c}{r^{1/(p-1)}}\qquad\mbox{for all }r>r_1$$
for suitable $r_1>r_0$ and some constant $c>0$. The limit \eqref{pol} may now be obtained as in the supercritical case, see Lemma \ref{theta}
and the subsequent arguments.

\section{\bf Subcritical case: proof of Theorem \ref{osc}}\label{oscillation}

Consider  the set $P$ of initial values $u(0)>0$ corresponding to positive solutions. By continuous dependence, $P$ is closed as a subset of $(0,+\infty)$, hence   its complement $P^c$ in $(0,+\infty)$   is open. Moreover we have just proved that $(0,U(0)]\subseteq P$, see Theorem \ref{polynomial}. As a first main result of this section we want to prove equality of such two sets.

The set $P^c$ coincides with   the set  of $\alpha=u(0)$ such that the corresponding radial solution $u$ to \eqref{cauchy} changes sign.   This set is  contained in $(U(0),+\infty)$. Let us  prove that it contains all the large values of $\alpha$.

\begin{lemma}\label{rescale}
Let $u$ be the solution to   \eqref{cauchy} with $1< p< \frac{n+2}{n-2}$,
corresponding to an initial datum $u(0)=\alpha$ with $\alpha>
\alpha_0$ large. Then there exists a first point $r_0>0$ depending on $a$ where $u(r_0)=0$. Moreover, $r_0(\alpha)\to0$ as $\alpha\to+\infty$ and $u$ intersects $U$ exactly once in $[0,r_0]$.
\end{lemma}
\begin{proof} We use a blow-up method. Let $u_\lambda$ be the solution of \eqref{cauchy} with $u_\lambda(0)=\alpha=\lambda^{2/(p-1)}$. Define
\neweq{changeof}
v_\lambda(\lambda r)=u_\lambda(r)\lambda^{-\frac{2}{p-1}}\,,\qquad\mbox{so that}\qquad
v_\lambda(0)=1\,.
\endeq
Setting $s=\lambda r$, we obtain that $v_\lambda$ satisfies the equation
\[
v''_\lambda(s)+\frac{n-1}{s}\,\left(\coth \frac s\lambda\right)\frac s\lambda\, v'_\lambda(s)+ |v_\lambda(s)|^{p-1}v_\lambda(s) =0\,.
\]
Let $S>0$ and take the limit $\lambda\to \infty$, so that for  $0\le s\le S$ we have $\coth\left(s/\lambda\right)(s/\lambda)\to 1$   uniformly   and the solution $v_\lambda$ tends to the solution $\overline{v}$ of the equation
\[
\overline{v}''(s)+\frac{n-1}{s}\overline{v}'(s)+ |\overline{v}(s)|^{p-1}\overline{v}(s) =0\,,\qquad\mbox{with}\qquad \overline{v}(0)=1\,,
\]
and the convergence of $v_\lambda\to \overline{v}$ holds in $C^1([0,S])$.   Although   the equation is singular at $s=0$, the singularity is associated to the radial   Euclidean   Laplacian and there is a regular branch of solutions such that $v'=0$ at $s=0$. Such solutions   depend smoothly on the coefficients and the data, and $0<S<\infty$ is still to be chosen.

Choose $S$ as bit larger than the first point $S_0$ for which $\overline{v}(S_0)=0$ (we also know that $\overline{v}'(S)<0$), hence, by the $C^1([0,S])$-convergence we can prove that also $v_\lambda(s)$ crosses transversally the $s$-axis, in a point $s_\lambda$ close to $S_0$. Going back to the original  variables, we have proved that $u_\lambda(r)$ crosses transversally the $r$-axis when the initial datum $u_\lambda(0)=\lambda^{2/(p-1)}$ is large, at a point $r_0$ which is approximatively $S_0/\lambda$. This also shows that the last claim holds true, indeed $r_0=O(1/\lambda)$. As long as crossing of the axis is transversal $r_0$ is a $C^1$ function of $a$. \end{proof}

\begin{lemma} $P^c$ is connected so that there exists $A\ge U(0)$ such that if $u$ is a radial solution to \eqref{triangle} satisfying $u(0)>A$ then it is sign-changing, whereas if $0<u(0)\le A$ then $u$ is positive.
\end{lemma}
\begin{proof} Suppose by contradiction that $P^c$ contains a maximal connected interval $(a,b)$ with $b>a>U(0)$. Then the radial solutions corresponding to $u(0)=a$
or $u(0)=b$ are everywhere positive. Continuity with respect to the initial datum then implies that the first zero of solutions corresponding
both to initial data approaching $a$ from above and $b$ from below must tend to $+\infty$. This violates uniqueness of the solution to the Dirichlet problem on balls, as proved in \cite[Proposition 4.4]{mancini}. Therefore, $P^c$ is connected as claimed. \end{proof}

We want to prove  that $A=U(0)$. To prove this fact we shall need several intermediate results.

For any two positive solutions  $u_\al$ and $u_\beta$ defined in some interval $(0,R)$ and satisfying $u_\alpha(0)=\alpha$, $u_\beta(0)=\beta$ we study the sign of the difference $w= u_\al-u_\beta$, which satisfies the equation
$$
w''+(\coth r)w'+ b(r)w=0, \qquad b(r)=\frac{ u_\al(r)^p-u_\beta(r)^p}{ u_\al(r)-u_\beta(r)}=p\,\tilde u^{p-1}(r)
$$
where $\tilde u(r)$ is an intermediate value between $u_\alpha(r)$ and $u_\beta(r)$. We have $b(r)>0$.

\begin{lemma}\label{lemma1} Let $\alpha_1>\alpha_2\ge \alpha_3>\alpha_4>0$. Then the first intersection between $u_{\alpha_1}$ and $u_{\alpha_2}$ cannot take place after the first intersection between $u_{\alpha_3}$ and $u_{\alpha_4}$.
\end{lemma}
\begin{proof} Let $w_1=u_{\alpha_1}-u_{\alpha_2}$ and $w_2=u_{\alpha_3}-u_{\alpha_4}$.
We have $w_1(0), w_2(0)>0$ and $w'_1(0)=w'_2(0)=0$. As long as there is no zero of $w_1$ and $w_2$, say for $0\le r\le r_1$, they satisfy
$$
w_1''+(\coth r)\,w'_1+ b_1(r)\,w_1=0, \quad w_2''+(\coth r)\,w_2'+ b_2(r)\,w_2=0\,,
$$
and the assumptions imply that $b_1>b_2$ in $[0,r_1]$. A Sturm-type Theorem then implies that $w_1$ must vanish at least when $w_2$ vanishes for the first time, or before. In fact, the equation satisfied by the quotient $z=w_1/w_2$ is
$$
w_2\,z''+(2w_2'+(\coth r)w_2)\,z'=-z(b_1-b_2)w_2\le 0,
$$
and moreover $z(0)>0$ and $z'(0)=0$. Integrating once this inequality it follows that $z'(r)\le 0$,
hence $z\le z(0)$ for $0<r<r_1$, which implies that $z_2$ cannot vanish before $z_1$.\end{proof}

Given a continuous function $u = u(r)$ defined on a closed interval $I\subset \R$, the positive and negative sets of $u$ are defined as follows:
$$
\Omega_u^+ = \{ r\in I: u(r) > 0\}, \quad \Omega_u^- = \{ r \in I: u(r) < 0\}
$$
A component of $\Omega_u^+$ (or $\Omega_u^-$) is a maximal open connected subset.

We define the number of sign changes of $u$ in $I$ as the number (finite or infinite) of connected components of $\{r : u(r)\ne 0\}$
minus one. This is also known briefly as the {\it lap number} of $u$ in $I$,  and is denoted by $Z(u,I)$. Alternatively, $Z(u,I)$ is the supremum of the integer numbers $k$ such that  there exist $k + 1$ points from $I$ such that
$r_0 < r_1 < \dots < r_k$, satisfying
$$
u(r_j)\,\cdot \,u(r_{j+1}) < 0 \quad \forall j = 0,1,\dots k-1\,.
$$
We will apply the known facts of the theory of lap numbers (see \cite{G, V1}) to the difference of two solutions of the Emden-Fowler equation defined in an interval $I=[0,R]  \subset [0,\infty)$. One of the solutions in this application will be $U$, the other one $u_\alpha$ for some $\alpha>0$, $\alpha\ne U(0)$. We call
$$
v_\alpha(r)=u_\alpha(r)-U(r)
$$
Let $I_\alpha$ be the closure of the first connected component of $\Omega_{u_\alpha}^+$ and let
$Z_\alpha:= Z(v_\alpha, I_\alpha)$.

\begin{lemma} For $0<\alpha< U(0)$ and $\alpha>A$ we have $Z_\alpha=1.$
\end{lemma}

\begin{proof}(i) For $\alpha<U(0)$ we know that $I_\alpha=[0,\infty)$ and the set where $v_\alpha\ne 0$ has two connected components, a negative one near $r=0$ and a positive one for large $r$, separated by a unique zero $r_\alpha$. Hence, $Z_\alpha =1$ for all $\alpha\in (0,U(0))$.

\noindent (ii) Let us now consider $\alpha>A$. We have proved in Lemma \ref{rescale} that for all $\alpha\ge\alpha_0$ large enough, $u_\alpha$ crosses $U$
transversally at some small $r_\alpha$ and after that  $v_\alpha$ is negative until $u_\alpha$ becomes zero at some $R_\alpha>0$, also small. This means that
$$
Z_\alpha:= Z(v_\alpha, [0,R_\alpha])=1 \quad \forall \alpha\ge\alpha_0.
$$

\noindent (iii) We prove now that $Z_\alpha$ does not change in value when we let $\alpha\downarrow A$. Indeed, for all $\alpha\in (A,\infty)$ $v_\alpha$ is positive at the beginning  and negative at the end of $I_\alpha$ so that $Z_\alpha\ge 1$. The fact that $Z_\alpha<\infty$ is obtained by contradiction at a limit point of the set of zeros, since at that point $v_\alpha$ must have horizontal tangent, which contradicts the local uniqueness of solutions of the Cauchy problem \eq{cauchy}.

Next, we note that an increase in the number of connected components as $\alpha$ decreases cannot take place near the ends  of the defining interval. But in the middle of the interval an increase of $Z_\alpha$ means a new small interval of positivity and another one of negativity, hence $Z_\alpha$ is always odd. The supremum $\alpha_0$ of the
$\alpha$ for which $Z_\alpha>1$ will have a zero of $v_\al$ with horizontal tangent any point which is limit of the connected component that is lost as $\alpha\to\alpha_0$. Again, this goes against the local uniqueness of solutions of the Cauchy problem associated to the differential equation considered. Therefore, $Z_\alpha=1$
for all $\alpha>A$.
\end{proof}

We can now prove

\begin{lemma} One has $A=U(0)$.
\end{lemma}
\begin{proof} Assume that $A>U(0)$. The result that $U_A$ has at least one intersection with $U$ comes from Lemma \ref{lemma1},  applied to $\alpha_1=A$, $\alpha_2=\alpha_3=U(0)$ and $0<\alpha_4<U(0)$, since we know that the last two cross once
at a finite distance, hence $U_A$ must cross $U$ before. The result that it does not have two or more intersections is proved by contradiction: if there is a later point at which $v_A(r_1)>0$ using continuity of the solutions $Z_\alpha$ would be larger than 1 for $\alpha=A+\ve $ near $A$. Hence, $Z(A)=1$.

But then, since $U_A$ cannot change sign, it will be less than $U$ for all $r$ sufficiently large, and this is a contradiction with the uniqueness of global positive fast-decaying solutions.\end{proof}

The just proved lemma shows that $\alpha=U(0)$ is the threshold between positive solutions and sign-changing solutions
to \eq{cauchy}. Let us now prove the four Items in Theorem \ref{osc}.

Let $R(a)$ denote the first zero of radial solutions corresponding to $u(0)=a>U(0)$. By \cite[Proposition 4.4]{mancini} we know that the corresponding Dirichlet problem admits a unique radial positive solution in any ball of finite radius. This fact and the above results show that $a\mapsto R(a)$ is strictly decreasing with $\underset{a\to U(0)}\lim R(a)=+\infty$ and $\underset{a\to +\infty}\lim R(a)=0$. This proves Item $(i)$

In order to prove Item $(ii)$ we argue by contradiction assuming that $u$ oscillates infinitely many times. By Lemma \ref{zero},
for any $\varepsilon>0$ there exists $r_0$ such that $|u(r)|<\varepsilon^{1/(p-1)}$ for all $r>r_0$. Let $r_1>r_0$ be a point such
that $u(r_1)=0$, let $r_2>r_1$ the first maximum of $u$ after $r_1$ and $r_3>r_2$ the next zero of $u$, so that $u>0$ in $[r_2,r_3)$.
For any $r\in(r_2,r_3]$ we also have $u^\prime(r)<0$ in view of \eqref{alternative}. Since $\coth r>1$ for all $r>0$, from \eqref{equazione} we thus get
\[
u^{\prime\prime}+(n-1)u^\prime+\varepsilon u\ge0\ \ \ \forall r\in[r_2,r_3].
\]
We can and shall assume now that $\varepsilon<(n-1)^2/4$. Put
\neweq{lambda}
\lambda_1=\frac{n-1-\sqrt{(n-1)^2-4\varepsilon}}{2}\, ,\qquad \lambda_2=\frac{n-1+\sqrt{(n-1)^2-4\varepsilon}}{2}
\endeq
so that $\lambda_2>\lambda_1>0$. In the new variable $s=r-r_2$ the differential inequality satisfied by $V(s)=u(r-r_2)$ reads
\[
\left[e^{(\lambda_2-\lambda_1)s}\left(e^{\lambda_1s}V(s)\right)^\prime\right]^\prime\ge0,
\]
with $V$ satisfying the initial conditions $V(0)=B>0$, $V^\prime(0)=0$. Integrating from $0$ to $s$ yields:
\neweq{yields}
e^{(\lambda_2-\lambda_1)s}\left(e^{\lambda_1s}V(s)\right)^\prime\ge \lambda_1V(0)+V^\prime(0)=\lambda_1B
\endeq
or, equivalently, $\left(e^{\lambda_1s}V\right)^\prime\ge \lambda_1Be^{(\lambda_1-\lambda_2)s}$. Integrating again from $0$ to $s$ yields:
\[
e^{\lambda_1s}V(s)-B\ge \frac{\lambda_1B}{\lambda_2-\lambda_1}\left[1-e^{-(\lambda_2-\lambda_1)s}\right].
\]
This latter inequality can also be written as
\[
V(s)\ge \frac B{\lambda_2-\lambda_1}\left[\lambda_2e^{-\lambda_1s}-\lambda_1e^{-\lambda_2s}\right]
\]
or, in the original variable $r$,
\begin{equation}\label{v-decay}
u(r)\ge \frac B{\lambda_2-\lambda_1}\left[\lambda_2e^{-\lambda_1(r-r_2)}-\lambda_1e^{-\lambda_2(r-r_2)}\right]\ \ \ \ \forall r\in[r_2,r_3].
\end{equation}
But the r.h.s.\ of the latter inequality is positive since $\lambda_2>\lambda_1$. This means that $u(r_3)>0$, a contradiction. Therefore,
$u(r)$ never vanishes for $r>r_2$ and hence its number of zeros is finite. This proves Item $(ii)$.\par
In fact, the above arguments allow to prove the following stronger statement:

\begin{lemma}\label{ultimozero}
Let $u$ be a sign-changing solution to \eqref{cauchy} and assume that there exists $r_2>0$ such that
\neweq{r2}
u^\prime(r_2)=0\quad\mbox{and}\quad\varepsilon:=|u(r_2)|^{p-1}<\frac{(n-1)^2}{4}.
\endeq
Then $u$ does not vanish on $[r_2,+\infty)$ and there exists $C>0$ such that
\beq\label{general}
|u(r)|\ge C e^{-\lambda_1 r}\qquad\forall r\ge r_2,
\eeq
where $\lambda_1$ is as in \eqref{lambda}.
\end{lemma}
\begin{proof} This result was proved above in the case where $r_2$ is a relative maximum (at positive level) for $u$.
The case where $r_2$ is a relative minimum (at negative level) for $u$ can be treated similarly.
Indeed, suppose that $r_2$ is a local minimum of $u$ such that $\varepsilon:=|u(r_2)|^{p-1}<\frac{(n-1)^2}{4}$. For contradiction, assume that
there exists $r_3>r_2$ such that $u(r_3)=0$. We can use the fact that $u$ is negative in $[r_2,r_3)$ and that $u^\prime$ is positive in $(r_2,r_3]$
to conclude that
\[
u^{\prime\prime}(r)+(n-1)u^\prime(r)+\varepsilon u(r)\le0\qquad\forall r\in[r_2,r_3].
\]
Proceeding exactly as before we can conclude that $u(r_3)<0$, a contradiction. Hence $u$ remains negative for all $r>r_2$. Moreover, for a
suitable $B>0$, the bound
$$u(r)\le -\frac B{\lambda_2-\lambda_1}\left[\lambda_2e^{-\lambda_1(r-r_2)}-\lambda_1e^{-\lambda_2(r-r_2)}\right]\qquad\forall r\in[r_2,r_3]
$$
holds true. This also proves \eqref{general}.\end{proof}

\begin{figure}[ht]
\centering
\includegraphics[height=6cm, width=12cm]{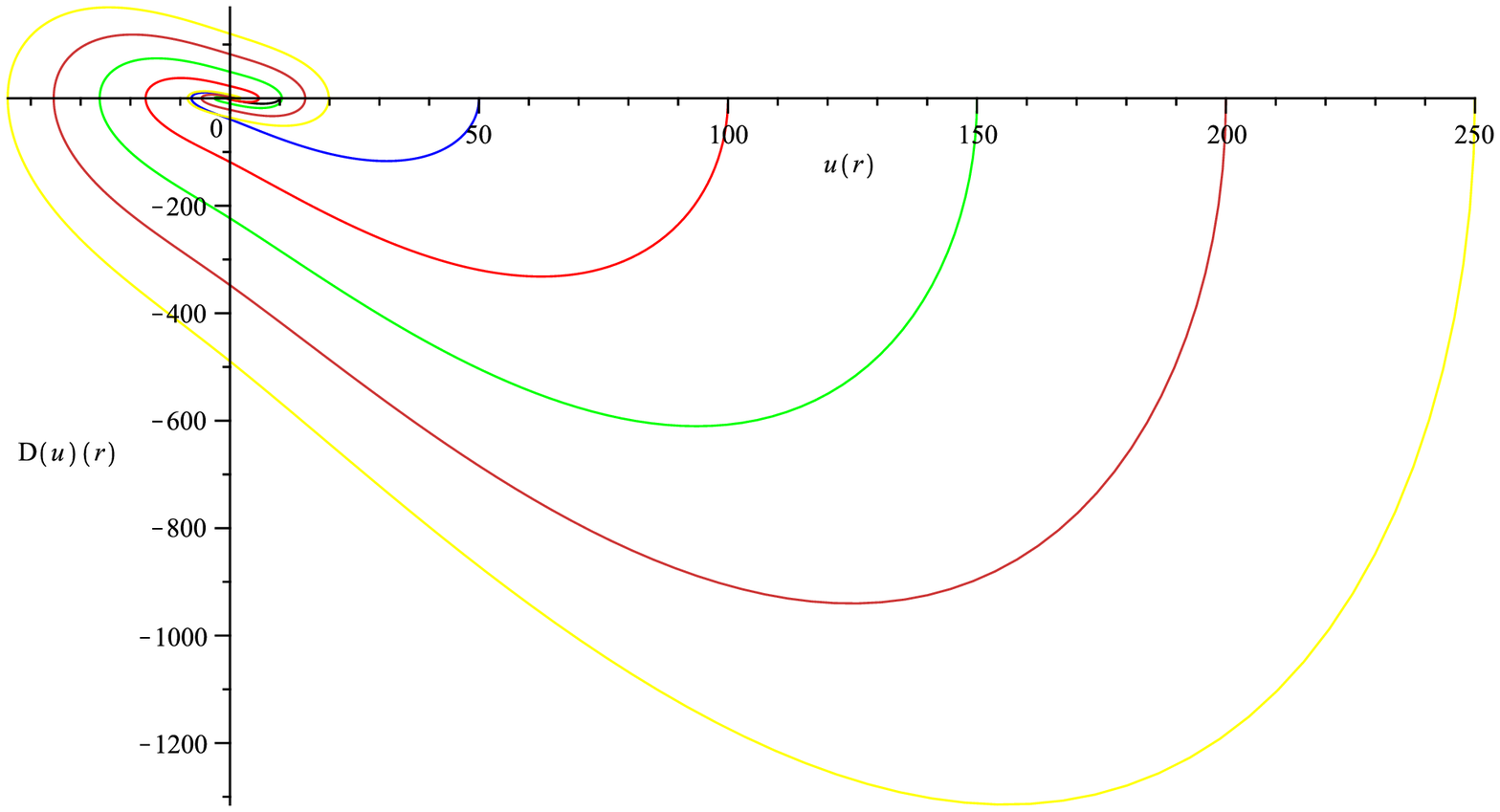}
\flushleft \color{darkred}\noindent\textit{Figure 5: phase plot of some sign-changing solutions when $d=3$, $p=2$ (subcritical case).}
\end{figure}

We are now ready to prove Item $(iii)$. For any $\delta>0$ consider the radial solution $u$ to \eqref{cauchy} with initial data $\alpha=U(0)+\delta$
and let $R(\alpha)$ denote the first zero of $u$. By $C^1$ continuous dependence w.r.t.\ to initial data, we see that
$$
R(\alpha)\to+\infty\quad\mbox{and}\quad u^\prime(R(\alpha))\to0\qquad\mbox{as }\delta\to0.
$$
The Lyapunov functional $F$ defined in \eqref{F} satisfies $F(R(\alpha))=u^\prime(R(\alpha)^2/2$.
Since $F$ is decreasing this implies that,
if we denote by $r_2$ the first minimum of $u$ after $R(\alpha)$, $|u(r_2)|<[(p+1)u^\prime(R(\alpha))^2/2]^{1/(p+1)}\to0$ as $\delta\to0$.
Provided $\delta$ is sufficiently small, the above steps then imply that \eqref{r2} holds. Hence, Lemma \ref{ultimozero} shows that for any such
$\delta$ \eqref{general} holds. In view of \eqref{lambda} we know that $\lambda_1\to0$ as $\delta\to0$. This proves that for all $\delta>0$
sufficiently small the radial solutions corresponding to $u(0)=\alpha=U(0)+\delta$  do not belong to L$^2({\mathbb H}^n)$,
hence they do not belong
to H$^1({\mathbb H}^n)$ as well. Moreover, again by Lemma \ref{ultimozero}, these solutions only have one zero.\par
In order to complete the proof of Item $(iii)$, we still have to prove \eqref{negativelimit}. To this end, consider one of the solutions just found,
namely a radial sign-changing solution $u$ to \eqref{cauchy} passing from positive to negative at some $r_0$, being the unique zero of $u$.
Since $|u(r)|<\varepsilon$ sufficiently small for all $r>r_0$, we may proceed as in the proof of \eqref{yields}. Setting $s=r-r_1$ and
$V(s)=u(r-r_1)$ with $r_1>r_0$ the first (unique) minimum of $u$, we find that
\[e^{(\lambda_2-\lambda_1)s}\left(e^{\lambda_1s}V(s)\right)^\prime\le \lambda_1V(0)+V^\prime(0)=-\lambda_1B\]
for $B=-V(0)>0$. This inequality implies that
\neweq{boundV}
V^\prime(s)\le -\lambda_1V(s)-Be^{-\lambda_2s}<-\lambda_1V(s).
\endeq
We now go back to the proof of Lemma \ref{theta}. It was proved there that $\Theta(r):=u^\prime(r)/u(r)\to0$ as $r\to+\infty$ in
the supercritical case and for positive solutions only. The proof was based on the bound \eqref{second lower}. A careful investigation
of the proof of Lemma \ref{theta} shows that it relies on the fact that $\frac{2(n-1+\eps)}{p+1}<n-1$ for $\varepsilon$ small.
One can check that Lemma \ref{theta} still holds for negative solutions satisfying $u^\prime(r)\le -\nu u(r)$ for all $r$
sufficiently large and a $\nu<n-1$. In the present situation, \eqref{boundV} shows that $u^\prime(r)\le-\lambda_1 u(r)$ for all $r\ge r_1$.
Since $\lambda_1$ is given by \eqref{lambda} we see that $\lambda_1<n-1$ provided $\varepsilon$ is small.
Therefore, Lemma \ref{theta} applies and $u^\prime(r)=o(u(r))$ as $r\to+\infty$. Since the argument outlined just
after the end of the proof of Lemma \ref{theta} depends exactly on the fact that $u^\prime(r)=o(u(r))$ as $r\to+\infty$, with the same arguments
used there we conclude that \eqref{negativelimit} holds. This completes the proof of Item $(iii)$.\par

In order to prove Item $(iv)$ we remark that Lemma \ref{ultimozero} shows that {\em zeros of $u_\alpha$ may enter from infinity once at a time
as $\alpha=u_\alpha(0)$ increases}. A further zero may only enter when the last critical point of $u_\alpha$ violates \eq{r2}.
So, since for $u_\alpha(0)<U(0)$ we have no zeros at all, it suffices to prove that the number of zeros of $u_\alpha$ can be arbitrarily large.
To see this, we make use of the same notations in the proof of Lemma \ref{rescale}. Since $\overline{v}$ has infinitely many zeros in view of
\cite[Theorem 15]{PuSer}, we know that for any integer $k$ there exists $S_k>0$ such that $\overline{v}$ has exactly $k$ zeros in $[0,S_k]$ and $\overline{v}(S_k)\neq0$.
The convergence $v_\lambda\to\overline{v}$ in $C^1([0,S_k])$ shows that also $v_\lambda$ has exactly $k$ zeros in
$[0,S_k]$ provided $\lambda$
is sufficiently large. In turn, by \eq{changeof} also $u_\lambda$ has exactly $k$ zeros in the interval $r\in[0,S_k/\lambda]$.
Therefore, $u_\lambda$ has at least $k$ zeros, provided $\lambda$ is large enough.

\begin{figure}[ht]
\centering
\includegraphics[height=6cm, width=12cm]{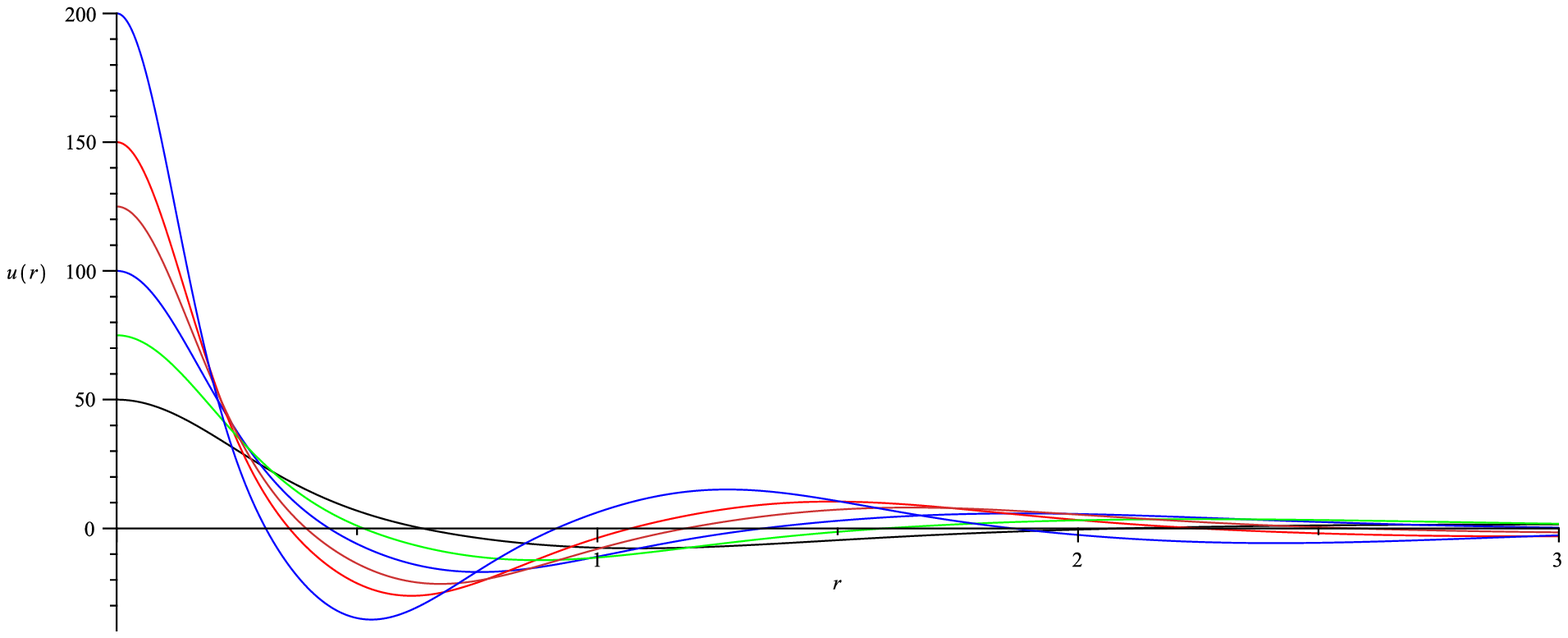}
\flushleft \color{darkred}\noindent\textit{Figure 6: plot of some sign-changing solutions when $d=3$, $p=2$ (subcritical case).}
\end{figure}

For the proof of Item $(v)$ we consider radial solutions $u$ in H$^1(\mathbb H^n)$ and notice that, by the proof of \cite[Theorem 3.1]{bhakta},
it follows that $|u|$ satisfies a uniform exponential upper bound. Both the proof of Lemma \ref{upperbound} and the proof of Lemma \ref{exp}
can be repeated with no change besides integrating over $(r_1,r)$ instead of over $(0,r)$, where $r_1$ is the last stationary point of $u$.
Then there exists $c>0$ such that $|u(r)|\le c\,e^{-(n-1)r}$ for all $r$. Passing from this to \eqref{n-1} can be done as in the comments just after
the proof of Lemma \ref{upperbound}. This completes the proof of Item $(v)$  and of Theorem \ref{osc}.\endproof

\section{\bf Sublinear case: proof of Theorem \ref{sublinear}}\label{proofsublin}

For contradiction, assume that there exists a local solution $u$ to \eqref{cauchy} which can be extended to a positive solution on $[0,+\infty)$.
By \eqref{alternative} we have
\beq\label{div}
-\Big[(\sinh r)^{n-1}\, u'(r)\Big]'=u(r)^p\, (\sinh r)^{n-1}\qquad\mbox{for all }r>0.
\eeq
By integrating this inequality and taking into account that $u$ is decreasing we obtain
\neweq{pminore}
-(\sinh r)^{n-1}\, u'(r)=\int_0^r u(s)^p\, (\sinh s)^{n-1}\,{\rm d}s\ge u(r)^p\phi_n(r)\qquad\mbox{for all }r>0.
\endeq
where $\phi_n$ is as defined in \eqref{phin}. Hence,
$$
-u'(r)u(r)^{-p}\ge\frac{\phi_n(r)}{(\sinh r)^{n-1}}\qquad\mbox{for all }r>0.
$$

By integrating over $(0,r)$ and recalling that $0<p<1$ we infer
$$-u(r)^{1-p}+u(0)^{1-p}\ge(1-p)\int_0^r\frac{\phi_n(t)}{(\sinh t)^{n-1}}\,{\rm d}t.$$
By letting $r\to+\infty$ and by recalling \eqref{triviallimit} we see that the l.h.s.\ tends to $u(0)^{1-p}$ whereas the r.h.s.\ tends to $+\infty$, contradiction.

\begin{figure}[ht]
\centering
\includegraphics[height=6cm, width=12cm]{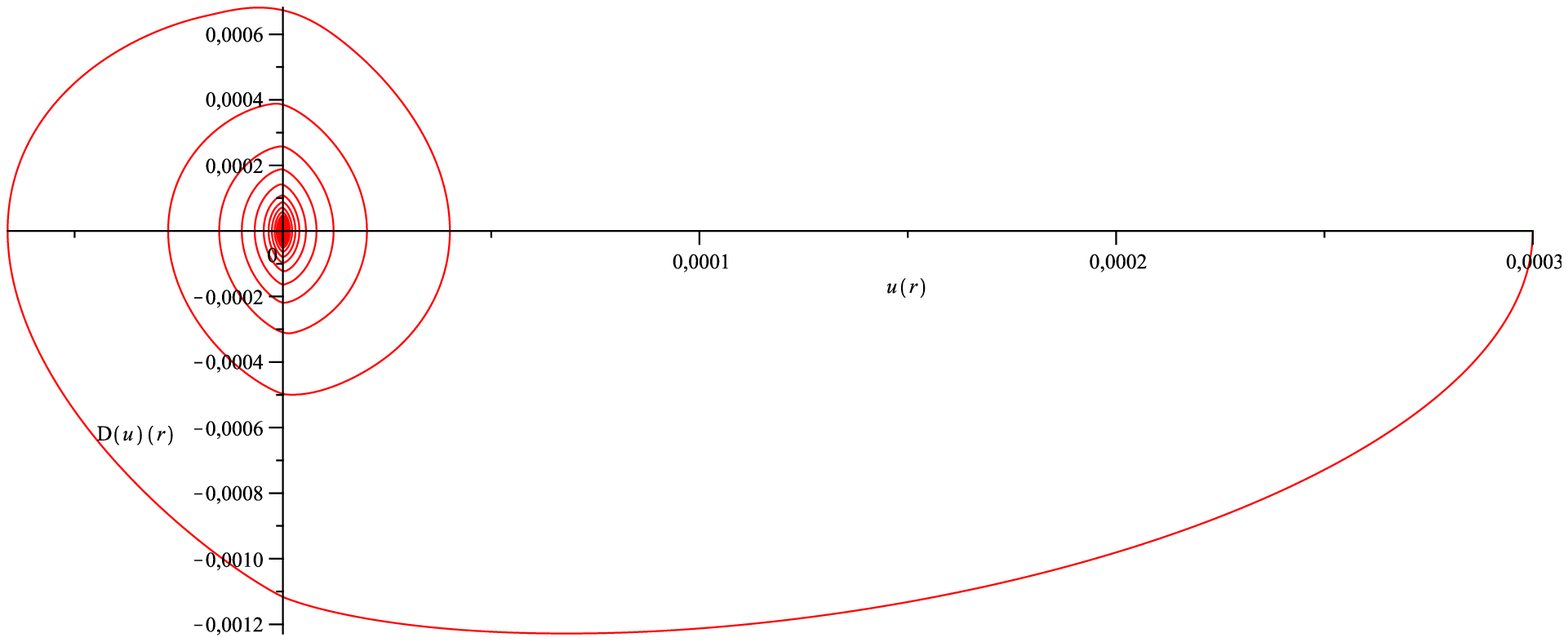}
\flushleft \color{darkred}\noindent\textit{Figure 7: phase plot of one sign-changing solution when $d=3$, $p=\frac12$ (sublinear case).}
\end{figure}

The same argument can be used to show that any solution admits infinitely many oscillations. Assume for contradiction that $u$ admits a local
maximum at some $r_0$ and remains positive for all $r\ge r_0$. Then integrate \eqref{pminore} over $(r_0,r)$ to reach the very same contradiction.
Similarly, one obtains a contradiction if $u$ admits a local minimum at some $r_0$ and remains negative for all $r\ge r_0$. Hence, any local
solution to \eqref{cauchy} admits infinitely many oscillations.

For the proof of \eqref{supinf}, by Remark \ref{psisub} we see that any solution $u$ satisfies $\Psi(r)>\delta$ for some $\delta>0$, provided
$r\ge r_\delta>0$. Hence, in any of the critical points $\rho\ge r_\delta$ of $u$, we have
$$\Psi(\rho)=\phi_n(\rho)\frac{|u(\rho)|^{p+1}}{p+1}>\delta$$
so that
$$|u(\rho)|>\left(\frac{\delta(p+1)}{\phi_n(\rho)}\right)^{1/(p+1)}\ge c e^{-\frac{n-1}{p+1}\rho}$$
and \eqref{supinf} follows.


\

\noindent {\large \sc Acknowledgments}. The authors acknowledge a contribution by the 2008 Spain-Italy research initiative HI2008-0178. The work was partially done while first and fourth authors were visitors at MSRI, Berkeley. These authors have been partially funded by Project MTM2008-06326-C02-01 (Spain).

\

\end{document}